\documentclass[aap]{imsart}
\RequirePackage{amsthm,amsmath,amsfonts,amssymb}
\RequirePackage[colorlinks,citecolor=blue,urlcolor=blue]{hyperref}
\RequirePackage{graphicx}
\usepackage{autobreak}
\usepackage{mathtools}
\usepackage{cite} 
\startlocaldefs
\numberwithin{equation}{section}
\theoremstyle{plain}
\newtheorem{theorem}{Theorem}[section]

\newtheorem{proposition}{Proposition}[section]
 
\theoremstyle{remark}

\newtheorem{definition}{Definition}[section]
\newtheorem{remark}{Remark}[section]
\newcommand\numberthis{\addtocounter{equation}{1}\tag{\theequation}}
\newcommand{\vertiii}[1]{{\left\vert\kern-0.25ex\left\vert\kern-0.25ex\left\vert #1 
    \right\vert\kern-0.25ex\right\vert\kern-0.25ex\right\vert}}
\DeclareMathOperator*{\esssup}{ess\,sup}
\DeclareMathOperator*{\essinf}{ess\,inf}
\newcommand{\f}{\mathcal{F}}
\newcommand{\g}{\mathcal{G}}
\newcommand{\rr}{\mathbb{R}}
\newcommand{\nn}{\mathbb{N}}
\newcommand{\pp}{\mathcal{P}}
\newcommand{\s}{\mathcal{S}}
\newcommand{\z}{\mathcal{Z}}
\newcommand{\p}{\mathbb{P}}
\newcommand{\q}{\mathbb{Q}}
\newcommand{\U}{U^\alpha(v)}
\newcommand{\Uu}{U^\alpha(\tau)}
\newcommand{\lo}{\mathbb{L}^\alpha_v}
\newcommand\normx[1]{\Vert#1\Vert}

\endlocaldefs

\begin{document}

\begin{frontmatter}
\title{Optimal Stopping Under Model Uncertainty in a General Setting}
\runtitle{Optimal Stopping Under Model Uncertainty in a General Setting}

\begin{aug}
\author[A]{\fnms{Ihsan}~\snm{Arharas}\ead[label=e1]{ihsan.arharas@lnu.se}}, 
\author[B]{\fnms{Siham}~\snm{Bouhadou}\ead[label=e2]{sihambouhadou@gmail.com}}, 
\author[C]{\fnms{Astrid}~\snm{Hilbert}\ead[label=e3]{astrid.hilbert@lnu.se}}
\and
\author[D]{\fnms{Youssef}~\snm{Ouknine}\ead[label=e4]{ouknine@uca.ac.ma}}
\address[A]{Department of Mathematics, Linnaeus University\printead[presep={,\ }]{e1}}
\address[B]{Department of Mathematics, Cadi Ayyad University\printead[presep={,\ }]{e2}}

\address[C]{Department of Mathematics, Linnaeus University\printead[presep={,\ }]{e3}}

\address[D]{Department of Mathematics, Cadi Ayyad University\printead[presep={,\ }]{e4}}

\end{aug}

\begin{abstract}
We consider the optimal stopping time problem under model uncertainty $R(v)=  \esssup\limits_{ \p \in \pp} \esssup\limits_{\tau \in \s_v} E^\p[Y(\tau) \vert \f_v]$, for every stopping time $v$, set in the framework of families of random variables indexed by stopping times. This setting is more general than the classical setup of stochastic processes, and particularly allows for general payoff processes that are not necessarily right-continuous. Under weaker integrability, and regularity assumptions on the reward family $Y=(Y(v), v\in \s)$, we show the existence of an optimal stopping time. We then proceed to find sufficient conditions for the existence of an optimal model. For this purpose, we present a “universal” Doob-Meyer-Mertens's decomposition for the Snell envelope family $R$ associated with $Y$ in the sense that it holds simultaneously for all $\p \in \pp$. This decomposition is then employed to prove the existence of an optimal probability model and study its properties. 
\end{abstract}



\end{frontmatter}

\section{Introduction}

Optimal stopping problems are a class of decision-making problems that involve finding the best time to take a particular action. These problems arise in various contexts where we need to make a decision based on uncertain information and limited resources, such as investment, job search, launching a new brand, exercising an option and others. All these problems naturally lead to an optimal stopping problem. Formally, the main target is to choose a random stopping time that maximizes the expected reward induced by the problem's payoff process. 

The classical solution to optimal stopping problems assumes that the reward's probability law is given to the decision maker. The problem is then formulated with respect to a unique probability measure. However, in numerous real-world scenarios, decision-makers contend with uncertainty regarding the true probabilistic model. As a result, the probability law that generates future rewards can be partially or completely unknown, creating what is known as model \textit{uncertainty} or \textit{ambiguity}. This can arise due to various reasons, such as incomplete information, parameter estimations errors, or the presence of hidden variables. An illustrative example is the analysis of American options in incomplete financial markets, where we have to deal with multiple equivalent martingale measures being uncertain about which one underlies the market. In such scenarios, the decision maker rely on a set of plausible probability measures or models, each of which could potentially result in vastly different optimal stopping strategies.

In the recent years, optimal stopping under ambiguity have attracted a lot of attention: \cite{Karat, Follmer,CherDel,Delbaen, Riedel,BayKarYao,BayYaoI,BayYaoII,ChengRiedel,Morlais}. This literature focuses on two types of optimal stopping problems: the worst-case and the best-case optimal stopping. The former$-$the so-called robust optimal stopping$-$maximizes the worst-case expected value: $\inf\limits_{\p \in \pp} E^\p[Y_\tau]$, via the choice of $\tau$; see, e.g. \cite{Riedel,BayYaoI,BayYaoII,BayYao2014
,ChengRiedel,nutzzhang}. The later, on the other hand, maximizes the best-case expected value: $\sup\limits_{\p \in \pp} E^\p[Y_\tau]$; see, e.g. \cite{BayYaoI,BayYaoII,Ekren,Belom,BayYao2017}. 
Here $\pp$ denotes the set of probability measures, and $Y$ represents the payoff process. 
 All the aforementioned works assumed that $\pp$ is dominated by a single probability measure, and $Y$ is an $\mathbb{F}$-adapted RCLL (càdlàg) process. A huge part of the literature focused on weakening the assumptions on the probability class $\pp$. Therefore, weak conditions have been discovered and improved with greater generality in the literature of mathematical finance; see e.g. \cite{Ekren,nutzzhang,nutz,BayYao2014}. However, there has been relatively little work in the literature allowing for more general payoff processes. To the best of our knowledge, the most general result given in the literature is that of Krätschmer \textit{et al.} \cite{Krat}: A numerically implementable method for single stopping problems under uncertainty in drift and jump intensity was proposed. For general reward processes driven by multi-dimensional jump-diffusions, we also refer to Roger \textit{et al.} \cite{Roger}.

 In this paper, we solve an optimal stopping problem under uncertainty with respect to a dominated probability class $\pp$, where the reward is given by a family $Y=(Y(\tau), \tau \in \s)$ of non-negative random variables indexed by stopping times. More precisely, letting $\s$ collect all stopping times with respect to a general filtration $\mathbb{F}$ on $\Omega$. In our main result, Theorem \ref{existence of optimal model}, we construct an optimal pair $(\p^*, \tau^*)\in \pp \times \s$ such that
 \begin{equation} \label{construct}
R \triangleq \sup_{ \p \in \pp} \sup_{\tau \in \s}  E^\p [Y(\tau)]=  \sup_{ \p \in \pp}  E^\p [Y(\tau^*)] = E^{\p^*}  [Y(\tau^*)] \quad \text{a.s.}
 \end{equation}
 The optimal stopping problem \eqref{construct} features generality along three dimensions: (i) the setup of families of random variables indexed by stopping times is more general that the classical setup of processes, therefore \eqref{construct} allows for more general payoff processes that are not necessarily right-continuous, (ii) particularly, for the existence of optimal stopping times, it allows for simpler proofs by using only classical tools of Probability Theory, and (iii) it allows to solve the problem under weaker assumptions than those used in previous literature.

Our technical set-up follows closely that of \cite{kobQuen}, which revisited the classical optimal stopping in the case of a reward given by a family of random variables indexed by stopping times. We extend their framework to account for model uncertainty \eqref{construct}. The key to this study is the generalized Snell envelope family of $Y$:
 \begin{align} \label{s1}
R(v)  \triangleq \esssup_{ \p \in \pp} \esssup_{\tau \in \s_v} E^\p[Y(\tau) \vert \f_v], \quad v\in \s,
 \end{align}
which is characterized as the smallest $\pp$-supermartingale family which is greater than the reward family $Y$. Within this setting, a family of random variables indexed by stopping times is called a $\pp$-supermartingale family if it is a $\p$-supermatingale family with respect to each measure $\p\in \pp$. (Please refer to Sections \ref{s2} for the notations). The properties of this generalized Snell envelope family follows naturally from the extant theory of optimal stopping in \cite{kobQuen}. In the classical literature, optimal stopping under uncertainty is typically formulated within the framework of processes. The reward is determined by a right-continuous, adapted process $(Y_t)$. The value function family $\lbrace R_v \rbrace_{v \in \s}$ is defined as above, and gives rise to a non-negative, adapted process $\mathcal{R}=\lbrace R_t, \mathcal{F}_t; 0 \leq t \leq T \rbrace$.  However, equation \eqref{s1} cannot be regarded as the process $\mathcal{R}$ evaluated at the stopping time $v$. The standard approach involves an important step known as the \textit{aggregation} step, that is, finding a RCLL modification $\mathcal{R}^0=\lbrace R^0_t, \mathcal{F}_t; 0 \leq t \leq T \rbrace$ of the process $\mathcal{R}$, such that for each $v\in \s$,  $R_v(\omega)=R^0_{v(\omega)}(\omega)$ for a.e. $\omega \in \Omega$. This process $\mathcal{R}^0$ is the generalized Snell envelope process of the payoff $Y$. It should be noted that this step is not trivial and relies on firm and sophisticated results of the general theory of stochastic processes (see e.g. \cite{Zam,Karat,aguilar}). The framework of our paper overcomes this complexity. One advantage of the set-up of families of random variables is that it allows to avoid the aggregation step. Another benefit is the weaker integrability assumption required for the existence of an optimal stopping time. Specifically, we only assume that $R<\infty$, which is a weaker condition compared to the classical assumption made in the literature, namely, that $E^\p[\sup\limits_{\tau \in \s} Y(\tau)] < \infty$, $\forall \p \in \pp$ (see e.g. \cite{Zam}). Moreover, optimal stopping times are characterized differently from the framework of processes. Given further assumptions on the reward family, we express the optimal stopping time as the essential infimum of a set of stopping times, instead of as the hitting time of processes, namely, $\tau^*= \essinf \lbrace \tau \in\s, R(\tau)=Y(\tau) \,\, \text{a.s.} \rbrace$.   

In order to prove the existence of an optimal model, we shall establish a “universal” Doob-Meyer-Mertens's decomposition for our Snell envelope family $\mathcal{R}=(R(v), v\in\s)$ in the sense that it holds simultaneously for all $\p \in \pp$. Precisely, under suitable conditions on the family $\mathcal{P}$, we decompose $\mathcal{R}$ as the difference between a $\pp$-martingale with RCLL paths, and an optional RCLL increasing process. There are two results which combined prove our claim. In the first, Theorem \ref{universal decompo}, we prove a universal optional decomposition theorem for $\pp$-supermartingale families. To the best of our knowledge, such a
decomposition was not obtained before. In Proposition \ref{decomp of R}, under the condition $\sup\limits_{\p \in \pp} E^\p[(Y^*)^2]<\infty$ on the random variable $Y^*:=\esssup_{\tau \in \s} Y(\tau)$, we show
that $\mathcal{R}$ satisfies the integrability condition of Theorem \ref{universal decompo}; thus, it admits the “universal” Doob-Meyer-Mertens's decomposition. All this leads to Theorem \ref{existence of optimal model}, the main result of Section \ref{optimmod}: there exists an optimal model $\p^*\in \pp$ such that \eqref{construct} holds.

Our framework, in particular, sheds new light on American knock-in barrier options valuation. Barrier options of knock-in type are financial derivatives that only become valuable if the underlying asset price reaches a predetermined barrier $H$, known as the “knocked-in” event. Once triggered, the option holder retains a standard American option with strike $K$, say, a put. For instance, an up-and-in barrier option requires the asset price to first rise and hit the upper barrier, triggering the knocked-in event. After this, the option holder would benefit from falling asset prices, resulting in a higher payoff for the put option.
Consider an underlying asset with the price process $S$ in a market environment where the riskless rate of return is $r$. Let $T$ be the expiration date for any option on the asset. Then, $(K- S)^+$ is the payoff of the put option when exercised at price $S$. In such a context, we are interested in the valuation problem for the American put-option of the “up-and-in” barrier type, with the payoff 
\begin{equation} \label{knockinpayoff}
    Y(t)= e^{-rt } (K- S(t))^+ \mathbf{1}_{\lbrace t \geq \tau_H \rbrace}, \quad 0 \leq t \leq T,
\end{equation}
where $\tau_H \triangleq \lbrace t \leq T : S(t) \leq H  \rbrace$ is the time where the option becomes “knocked-in” (see e.g. \cite{knockin}). The payoff process is not-right-continuous, our framework thus is particularly well-suited for analyzing this type of option. Now, assume that there is model uncertainty, and the decision maker has to choose a stopping time to maximize her expected reward under any of the models in a given probability class $\pp$. For instance, $\pp$ generated by an incomplete financial market. Therefore, the hedging price $H(S)$ of the American contingent claim \eqref{knockinpayoff} can be computed as the optimal expected reward in the following optimal stopping problem under uncertainty:
\begin{equation}
    H(S) \triangleq \sup\limits_{\p \in \pp} \sup\limits_{\tau \in \s} E^\p[e^{-r\tau } (K- S(\tau))^+ \mathbf{1}_{\lbrace \tau \geq \tau_H \rbrace}],
\end{equation}
$S=S_0$ the initial value. The interest of the setup of families of random variables indexed by stopping time has also been stressed by Kobylanski \textit{et al.} \cite{kobymulti}, in the framework of optimal multiple stopping. Further applications can be imagined, in for example, indifference valuation (seller’s perspective; Carmona \cite{carmona}, Roger \textit{et al.} \cite{Roger}) of general optimally stopped reward processes under the multiple priors model. 

The rest of the paper is organized as follows: Section \ref{s2} introduces the general set-up, including the formulation of the problem under model ambiguity. In Section \ref{s3}, we study the properties of the generalized Snell envelope family $(R(v), v\in \s)$. We also give necessary and sufficient conditions for the existence of an optimal pair $(\p^*, \tau^*)\in \pp \times \s$. Section \ref{s4} is devoted to establish the existence of an optimal stopping time. In Section \ref{optimmod}, we find conditions on the family $\pp$, under which there exists an optimal model for our problem.

\section{Setting and problem description} \label{s2}

 Consider a decision maker (e.g., seller, buyer, firm..) who needs to choose the best time to exercise a certain action in order to maximize her expected revenue. The decision is required to be made before a fixed predetermined time $T > 0$. Formally, we fix a filtered measurable base with finite horizon 
$(\Omega, \f, \mathbb{F}= \lbrace \f_t \rbrace_{ 0 \leq t\leq T})$. We assume that the filtration $\mathbb{F}$ satisfies the usual conditions of right continuity and augmentation by the null sets of $\f= \f_T$. By $\s$ we denote the class
of $\mathbb{F}$-stopping times with values in $[0,T]$. For a stopping time $v\in \s$, we define $\s_v:= \lbrace \tau \in \s \,\,  \vert \,\,  \tau \geq v \,\, \text{a.s.} \rbrace$ and $\s_{v^+}:= \lbrace \tau \in \s \,\,  \vert \,\,  \tau > v \,\, \text{a.s. on $\lbrace v <T \rbrace$ and $\tau = v$ a.s. on $\lbrace v =T \rbrace$} \rbrace$.

 The decision maker chooses a stopping time $\tau$ in $\s$. When she decides to stop at $\tau$, we assume that she receives the quantity $Y(\tau)$, where $Y(\tau)$ is a non-negative $\f_\tau$-measurable random variable. The problem is thus expressed in terms of families of random variables indexed by stopping times. The family $Y=(Y(\tau), \tau \in S)$ is called the reward (or payoff) family. Since the decision maker is uncertain about the probability with which the reward $Y$ is generated, she uses a class $\pp= \lbrace \p \,\vert \,\p \sim  \q \rbrace$ of probability models $\p$ on $(\Omega, \f)$ that share the same null sets with a base reference model $\q \in\pp$. Henceforth, the decision maker has to choose a random stopping time that maximizes the expected reward under any of the models. In our context, this yields to formulating the following optimal stopping problem (at time 0 and at time $v\in \s$):
 \begin{align*} \label{eq}
 &R \triangleq \sup_{ \p \in \mathcal{P}} \sup_{\tau \in \s} E^\p[Y(\tau)], \\
 &R(v) \triangleq \esssup_{ \p \in \pp} \esssup_{\tau \in \s_v} E^\p[Y(\tau) \vert \f_v]. \numberthis
 \end{align*} 
 Therefore, solving the optimal stopping time problem \eqref{eq} at time $v$, mainly consists to prove the existence of an optimal stopping time $\tau^*(v)$ and an optimal model $\p^*$, such that,
 \begin{equation}
 R(v) = E^{\p^*} [Y(\tau^*(v))\vert \f_v]\quad \text{a.s.}
\end{equation}  

We also use the following notations: 
\begin{enumerate}
\item[•] $L^0(\f, \rr)$ is the algebra of equivalence classes of $\rr$-valued random variables on $\Omega$.

\item[•] $L^0_+(\f,\rr):=\lbrace \xi \in L^0(\f, \rr) \vert  \xi \geq 0 \rbrace$.

\item[•] $L^1(\Omega, \f,\q)$ is the set of real-valued optional processes $\xi$ with $\sup\limits_{0\leq t\leq T} E^\q [\vert \xi_t\vert] < \infty$.

\item[•] For any sub-$\sigma$-field $\g$ of $\f$, for any probability measure $\p$ on $(\Omega, \f)$, we let $L^2(\g, \p)$ be the collection of real-valued $\g$-measurable random variables with absolute value admitting a $2$-moment under $\p$. This space is endowed with its usual norm.

\item[•] For a ladlag process $\phi$, we denote by $\phi_{t+}$ and $\phi_{t-}$ the right-hand and left-hand limit of $\phi$ at time $t$. We denote by $\Delta \phi_t := \phi_t - \phi_{t-}$ the size of left jump of $\phi$ at time $t$ (with the convention $\phi_{0^-}=\phi_0$), and by $\Delta_+ \phi_t := \phi_{t+} - \phi_t$ the size of right jump of $\phi$ at time $t$. 
\item[•] For real valued random variables $X$ and $X_n$, $n\in \nn$, “$X_n \uparrow X$” stands for “the sequence $(X_n)$ is non-decreasing and converges to $X$ a.s.”
\end{enumerate}

The following Definition can be found in \cite{L0conv}.

\begin{definition}\label{convexdef}
A subset $U$ of $L^1(\Omega, \f,\q)$ is said to be $L^0$-convex if $x \xi + (1 - x)\zeta \in U$ for all $\xi,\zeta \in U$ and $x \in L^0_+(\f,\rr)$ such that $0 \leq x \leq 1$.
\end{definition}

\section{First properties and necessary and sufficient conditions for optimality}\label{s3}

 In this subsection, we present some preliminary results on the value families $R$ and $R^+$ when the reward is given by an admissible family of random variables indexed by stopping times. 

\begin{definition} \label{admissible}
We say that a family $Y = (Y (\tau), \tau \in \mathcal{S})$ is admissible if it satisfies the
following conditions
\begin{enumerate}
\item For all $\tau \in \mathcal{S}$, $Y(\tau)$ is a $\mathcal{F}_\tau$-measurable non-negative random variable.
\item For all $\tau, \tau' \in \mathcal{S}$, $Y(\tau)= Y(\tau')$ a.s. on $\lbrace \tau = \tau'\rbrace$.
\end{enumerate}
If moreover for all $\tau \in \mathcal{S}$, $Y(\tau)$ is square-integrable, we say that the admissible family $Y$ is square integrable. 
\end{definition} 

\begin{remark}
 It is always possible to define an admissible family associated with a given process. More precisely, let $(Y_t)_{t\in [0,T]}$ be a non-negative progressive process. Define $Y(\tau) := Y_\tau$, for
each $\tau \in \s$. The family $Y = (Y_\tau, \tau \in \s)$ is clearly admissible.
\end{remark} 

Let now $(Y (\tau), \tau \in \mathcal{S})$ be an admissible reward family. For $v\in \mathcal{S}$, the value function at time $v$ is defined by 
\begin{equation}
 R(v):= \esssup_{ \p \in \pp} \esssup_{\tau \in \s_v} E^\p[Y(\tau) \vert \f_v], \label{eq1*}
 \end{equation} 
 the strict value function at time $v$ is defined by
 \begin{equation}
 R^+(v):= \esssup_{ \p \in \pp} \esssup_{\tau \in \s_v{^+}} E^\p[Y(\tau) \vert \f_v].
 \end{equation} 
 Due to our assumptions we can define the density process $Z^{\p}_t \triangleq \frac{d\p}{d\q}\vert_{\mathcal{F}_t}$, for any $t\in [0,T]$ and any $\p \in \mathcal{P}$. One can easily see that, $Z^\p$ is a $\q$-martingale, $Z^\p_0=1$ and $Z^\p_t>0$ a.s., for all $t\in[0,T]$.
  Hence, we define the set of $\q$-martingales:
  \begin{equation} \label{Z}
      \mathcal{Z}= \Big\lbrace  Z_t^\p=\frac{d\p}{d\q}\Big\vert_{\mathcal{F}_t}, \quad 0 \leq t \leq T \quad \text{for} \quad \p \in \mathcal{P} \Big\rbrace.
  \end{equation}
  
 For technical reasons, which we shall unveil in the following, we shall assume that the set $\mathcal{Z}$ is $L^0$-convex. From the Bayes' rule, we can rewrite 
 \begin{equation}
   E^\p[Y(\tau) \vert \mathcal{F}_v]= E^\q\Big[ \frac{Z^{\p}_\tau}{Z^{\p}_v} Y(\tau) \Big\vert \mathcal{F}_v \Big]= \frac{1}{Z^{\p}_v} E^\q[Z^{\p}_\tau Y(\tau) \vert \mathcal{F}_v], \quad \tau \in \mathcal{S}_v.
   \end{equation}
Therefore, $R(v)$ and $R^+(v)$ becomes 
\begin{align}
&R(v)= \esssup_{ Z \in \mathcal{Z}} \esssup_{\tau \in \mathcal{S}_v} \Gamma(v\vert \tau, Z), \label{R}\\
&R^+(v)=\esssup_{ Z \in \mathcal{Z}} \esssup_{\tau \in \mathcal{S}_{v^+}} \Gamma(v\vert \tau, Z), \label{R^+}
\end{align}
where $ \Gamma(v\vert \tau, Z):=  E^\p[Y(\tau) \vert \mathcal{F}_v] \triangleq \frac{1}{Z_v} E^\q[Z_\tau Y(\tau) \vert \mathcal{F}_v]$. Since this random variable depends only on the restriction of the process $Z$ to the stochastic interval $\textbf{[} v, \tau\textbf{]}$, we consider $\mathcal{Z}_{v,\tau}$ to be the restriction of $\mathcal{Z}$ to this interval.

\begin{proposition} (Admissibility of $v$ and $v^+$) \label{r r+ admiss}
The families $R= (R(v), v\in \mathcal{S})$ and $R^+= (R^+(v), v\in \mathcal{S})$ defined by \eqref{R} and \eqref{R^+} are admissible. 
\end{proposition}

\begin{proof}
Let us prove the property for $R^+= (R^+(v), v\in \mathcal{S})$. Thanks to the definition of the essential supremum (see Neveu \cite{Neveu}), one can see that for each $v \in \mathcal{S}$, $R^+(v)$ is an $\mathcal{F}_v$-measurable random variable. Let now $v,v' \in \mathcal{S}$ and set $A:= \lbrace v = v' \rbrace$. For every $\tau \in \mathcal{S}_{v^+}$, put $\tau_A= \tau \mathbf{1}_{A} + T \mathbf{1}_{A^c}$. We see at once that $\tau_A \in \mathcal{S}_{v^+}$. Since $A \in \mathcal{F}_v \cap \mathcal{F}_{v'}$, we get a.s. on $A$, 
 \begin{align*}
 \Gamma(v\vert \tau, Z)= \frac{1}{Z_v} E^\q[Z_\tau Y(\tau) \vert \mathcal{F}_v]&= \frac{1}{Z_v} E^\q[Z_{\tau_A} Y(\tau_A) \vert \mathcal{F}_{v'}]\\
 &= \Gamma(v'\vert \tau_A, Z) \\
 &\leq  \esssup_{ Z \in \mathcal{Z}} \esssup_{\tau \in \mathcal{S}_{v'^+}} \Gamma(v'\vert \tau, Z)= R^+(v').
\end{align*}   
Then, taking the essential supremum over $Z \in \mathcal{Z}$ and $\tau \in \mathcal{S}_{v^+}$, we obtain $R^+(v)\leq R^+(v')$ a.s. By symmetry of $v$ and $v'$, we obtain the converse inequality and the proof is complete. Similar arguments show that $R$ is admissible.
\end{proof}

\begin{proposition}(Optimizing sequence for $R$ and $R^+$)
For any $v \in \mathcal{S}$, the family of random variables $\lbrace \Gamma(v\vert \tau,Z) / \tau \in \mathcal{S}_v, Z\in \mathcal{Z}_{v,\tau} \rbrace$ (resp. $\lbrace \Gamma(v\vert \tau,Z) / \tau \in \mathcal{S}_{v^+}, Z\in \mathcal{Z}_{v,\tau} \rbrace$) is closed under pairwise maximization. That is, there exists a sequence $\lbrace (\tau_n, Z^n)\rbrace_{n \in \mathbb{N}}$ with $\tau_n$ in $\mathcal{S}_v$ (resp. $\mathcal{S}_{v^+}$) and $Z^n \in \mathcal{Z}_{v,\tau_n}$ such that the sequence $\lbrace \Gamma(v \vert \tau_n, Z^n) \rbrace_{n\in \mathbb{N}}$ is increasing and such that 
\begin{align*}
R(v) \quad \text{(resp. $R^+(v)$)}\quad = \lim\limits_{n \rightarrow \infty} \Gamma(v\vert \tau_n, Z^n) \quad \text{a.s.}
\end{align*}
\end{proposition}

\begin{proof}
We prove the property for $R$, the arguments are the same for $R$. Let $\tau_1, \tau_2 \in \mathcal{S}_{v}$ and $Z^1$, $Z^2 \in \mathcal{Z}$, and consider $A= \lbrace \Gamma(v\vert \tau_2, Z^2) \geq \Gamma(v \vert \tau_1, Z^1) \rbrace \in \mathcal{F}_v$. We put also
\begin{align*}
    \tau&= \tau_1 \mathbf{1}_{A^c} + \tau_2 \mathbf{1}_A,\\
    Z_t&= Z^1_t \q(A^c \vert \mathcal{F}_t) + Z^2_t \q(A \vert \mathcal{F}_t), \quad 0 \leq t \leq T.
\end{align*}
Then $\tau$ is a stopping in $\mathcal{S}_v$, and since $\mathcal{Z}$ is $L^0$-convex, we have $Z \in \mathcal{Z}$. Therefore, \begin{align*}
 \Gamma(v \vert \tau,Z)&= E^\q \Big [\frac{Z_\tau}{Z_v} Y(\tau) \Big\vert \mathcal{F}_v\Big]\\
 &= E^\q\Big[ \frac{Z^1_{\tau_1}}{Z^1_v} Y(\tau_1) \Big\vert \mathcal{F}_v \Big] \mathbf{1}_{A^c} +  E^\q\Big[ \frac{Z^2_{\tau_2}}{Z^2_v} Y(\tau_2) \Big\vert \mathcal{F}_v \Big] \mathbf{1}_{A}\\
 &= \Gamma(v \vert \tau_1,Z^1) \mathbf{1}_{A^c} + \Gamma(v \vert \tau_2,Z^2) \mathbf{1}_{A} \\
 &= \Gamma(v \vert \tau_1,Z^1) \vee \Gamma(v \vert \tau_2,Z^2).
 \end{align*}
 Hence, the set $\lbrace \Gamma(v\vert \tau,Z) / \tau \in \mathcal{S}_v, Z\in \mathcal{Z}_{v,\tau} \rbrace$ is closed under pairwise maximization. The existence of an optimizing sequence follows from a classical result on essential supremum (Neveu (1975) ).
\end{proof}

\begin{definition}
 An admissible family $Y=(Y(\tau), \tau\in \mathcal{S})$ is said to be a $\pp$-supermartingale family (resp. a $\mathcal{P}$-martingale family) if for all $\tau, \tau' \in \mathcal{S}$, such that $\tau\geq \tau'$ a.s.,
 \begin{align*}
 &E^\p[Y(\tau) \vert \mathcal{F}_{\tau'}] \leq Y(\tau') \quad \text{a.s. for each measure} \quad \p \in \mathcal{P}, \\
 &\text{(resp.} \quad E^\p[Y(\tau) \vert \mathcal{F}_{\tau'}] = Y(\tau') \quad \text{a.s. for all measure} \quad \p \in \mathcal{P}). 
 \end{align*}
 \end{definition}
 
 The following proposition states that the value function $R$ and the strict value function $R^+$ are both supermartingale families.
 
 \begin{proposition}
 The admissible families $R =(R(v), v \in \mathcal{S})$ and $R^+ = (R^+(v), v \in \mathcal{S})$ are $\mathcal{P}$-supermartingale families in the sense of the above Definition. Moreover, the value family $R =(R(v), v \in \mathcal{S})$ is characterized as the Snell envelope family associated with $(Y(v), v \in \mathcal{S})$, that is the smallest $\mathcal{P}$-supermartingale family which is greater a.s. than $(Y(v), v \in \mathcal{S})$.
 \end{proposition}
 
 \begin{proof}
 Let us prove the first property for $R$. Let $v,v' \in \mathcal{S}$ with $v \geq v'$ a.s. By Proposition 3, there exists an optimizing sequence $\lbrace (\tau_n, Z^n)\rbrace_{n \in \mathbb{N}}$ with $\tau_n$ in $\mathcal{S}_v$ and $Z^n \in \mathcal{Z}_{v,\tau_n}$.  Let now $\p\in \mathcal{P}$, by the monotone convergence theorem, we get
 \begin{align*}
 E^\p[R(v) \vert \mathcal{F}_{v'}]=   \lim\limits_{n \rightarrow \infty} E^\p[\Gamma(v\vert \tau_n, Z^n)  \vert \mathcal{F}_{v'}] \quad \text{a.s.}
  \end{align*}
  Since for each $n$, $\tau_n \in \mathcal{S}_{v'}$ and $Z^n  \in \mathcal{Z}_{v,\tau_n}$, we have 
  $$E^\p[\Gamma(v\vert \tau_n, Z^n)  \vert \mathcal{F}_{v'}] \leq R(v') \quad \text{a.s.}$$ Hence 
  $$ E^\p[R(v) \vert \mathcal{F}_v']  \leq R(v') \quad \text{a.s.,}$$
  which gives the $\mathcal{P}$-supermartingale property of $R$.
  What is left is to show the second statement. We see at once that $(R(v), v \in \mathcal{S})$ is a $\mathcal{P}$-supermartingale family and for each $v \in \mathcal{S}$, $R(v) \geq Y(v)$ a.s.  Let $(\bar{R}(v), v \in \mathcal{S})$ be another $\mathcal{P}$-supermartingale family such that for each $v \in \mathcal{S}$, $\bar{R}(v) \geq Y(v)$ a.s. Fix $v \in \mathcal{S}$. By the $\mathcal{P}$-supermartingale property of $\bar{R}$, we have for every stopping time $\tau \in\mathcal{S}_v$, and every measure $\p \in \mathcal{P}$
\begin{equation*}
\bar{R}(v) \geq E^\p[\bar{R}(\tau) \vert \mathcal{F}_v] \geq E^\p[ Y(\tau) \vert \mathcal{F}_v] \quad \text{a.s.}
\end{equation*}
Taking the supremum over $\tau \in \mathcal{S}_v$ and $\p\in \mathcal{P}$, we obtain $\bar{R}(v) \geq R(v)$ a.s., and the proposition follows.
 \end{proof}
 
  \begin{proposition} \label{R=YvR+}
 For every $v \in \mathcal{S}$,  $R(v) = Y(v) \vee R^+(v)$ a.s.
 \end{proposition}
 
 \begin{proof}
 Let $v$ be a stopping time in  $\mathcal{S}$. Take $\tau\in \mathcal{S}_v$, we first show that for every $Z \in \mathcal{Z}$,
    \begin{equation}
    \Gamma(v \vert \tau,Z) \leq Y(v) \vee R^+(v) \quad \text{a.s.}
    \end{equation}
    We set $\bar{\tau}=\tau \mathbf{1}_{\lbrace v> \tau \rbrace} + T \mathbf{1}_{\lbrace \tau= v \rbrace}$. One can see that $\bar{\tau}$ belongs to $\mathcal{S}_{v^+}$. Hence,
    \begin{equation}
    \Gamma(v \vert \tau, Z)  \mathbf{1}_{\lbrace \tau>v \rbrace} = \Gamma(v \vert \bar{\tau}, Z) \mathbf{1}_{\lbrace \tau> v \rbrace}\leq R^+(v) \mathbf{1}_{\lbrace \tau > v \rbrace} \quad \text{a.s.}
    \end{equation}
   Therefore, 
   $$\Gamma(v \vert \tau, Z)= Y(v) \mathbf{1}_{\lbrace  \tau= v \rbrace} +  \Gamma(v \vert \tau, Z) \mathbf{1}_{\lbrace \tau> v \rbrace} \leq Y(v) \mathbf{1}_{\lbrace  \tau= v \rbrace} +  R^+(v) \mathbf{1}_{\lbrace \tau > v \rbrace} \quad \text{a.s.} $$
   By taking the essential supremum over $\tau \in \mathcal{S}_v$ and then the essential supremum over $Z \in \mathcal{Z}$, we get $R(v) \leq Y(v) \vee R^+(v)$ a.s.
   The other inequality follows immediately from the fact that $R(v) \geq R^+(v)$ a.s. and $R(v) \geq Y(v)$ a.s., which completes the proof. \\ 
 \end{proof}
   
   \begin{proposition} \label{expectation of R} For any $v \in \mathcal{S}$, $\tau \in \mathcal{S}_v$ and $\p\in \mathcal{P}$, we have
\begin{equation} \label{prop5}
 E^\p[R^+(\tau) \vert \mathcal{F}_v]= \esssup_{ \sigma \in \mathcal{S}_{\tau^+} }  E^\p[Y(\sigma) \vert \mathcal{F}_v] \quad \text{a.s.}
 \end{equation} 
 In particular, $E^\p[R^+(\tau)]= \sup\limits_{\sigma \in \mathcal{S}_{\tau^+}} E^\p[Y(\sigma)]$.
   \end{proposition}
   
   \begin{proof}
   Let $\p\in \mathcal{P}$. Denote by $Z= Z^\p$. From Proposition 3, there exists a sequence $\lbrace (\tau_n, Z^n)\rbrace_{n \in \mathbb{N}}$ with $\tau_n$ in $\mathcal{S}_{\tau^+}$ and $Z^n \in \mathcal{Z}_{\tau,\tau_n}$, such that  
   \begin{align*}
R^+(v) = \lim\limits_{n \rightarrow \infty} \Gamma(\tau\vert \tau_n, Z^n) \quad \text{a.s.}
\end{align*}
We can suppose without loss of generality that $Z^n_u = Z_u$, $\forall u \in [v,\tau]$.
Using Fatou's lemma, we get
   \begin{align*}
   E^\p[R^+(\tau)\vert  \mathcal{F}_v]&= E^\q\Big[ \frac{Z_\tau}{Z_v} R^+(\tau) \Big\vert  \mathcal{F}_v \Big] \\
   &= E^\q \Big[ \frac{Z_\tau}{Z_v} \lim\limits_{n \rightarrow \infty} E^\q \Big[\frac{Z^n_{\tau_n}}{Z^n_\tau} Y(\tau_n) \Big\vert \mathcal{F}_\tau \Big] \Big\vert  \mathcal{F}_v \Big]\\
   &= E^\q\Big[ \lim\limits_{n \rightarrow \infty} E^\q\Big[ \frac{Z_\tau}{Z_v} \frac{Z^n_{\tau_n}}{Z^n_\tau} Y(\tau_n) \Big\vert \mathcal{F}_\tau \Big] \Big\vert  \mathcal{F}_v \Big]\\
   &\leq \lim\limits_{n \rightarrow \infty} E^\q \Big[\frac{Z^n_{\tau_n}}{Z^n_v} Y(\tau_n)  \Big\vert  \mathcal{F}_v \Big]\\
   &=  \lim\limits_{n \rightarrow \infty} E^\p[Y(\tau_n)  \vert  \mathcal{F}_v]\\
   &\leq \esssup_{ \sigma \in \mathcal{S}_{\tau^+} }  E^\p[Y(\sigma) \vert \mathcal{F}_v]. 
   \end{align*}
   Now from the $\mathcal{P}$-supermartinglae property of $R^+$, we have for all $\tau \in \mathcal{S}_v$ and all $\sigma \in \mathcal{S}_{\tau^+}$
   \begin{align*}
E^\p[Y(\sigma) \vert \mathcal{F}_\tau] \leq R^+(\tau) \quad \text{a.s.}
\end{align*} 
Thus, for each $\sigma \in \mathcal{S}_{\tau^+}$, we have 
\begin{align*}
E^\p[E^\p[Y(\sigma) \vert \mathcal{F}_\tau] \vert \mathcal{F}_v]= E^\p[Y(\sigma) \vert \mathcal{F}_v] \leq E^\p[R^+(\tau) \vert \mathcal{F}_v]  \quad \text{a.s.}
\end{align*}
By taking the essential supremum over $ \sigma \in \mathcal{S}_{\tau^+}$ we derive the reverse inequality:
\begin{align*}
\esssup_{ \sigma \in \mathcal{S}_{\tau^+} }  E^\p[Y(\sigma) \vert \mathcal{F}_v] \leq E^\p[R^+(\tau) \vert \mathcal{F}_v ] \quad \text{a.s.}
\end{align*}
The proof is complete.\\
   \end{proof}
      
   We now give a crucial property of regularity for the strict value function family, namely the $\pp$-right continuity along stopping times in expectation. Let us first introduce the following definition.
   
\begin{definition}
   An admissible family $Y= (Y(\tau), \tau \in \mathcal{S})$ is said to be $\pp$-right continuous in expectation ($\pp$-RCE) if for every $\tau \in \mathcal{S}$ and for any sequence of stopping times $(\tau_n)_{n\in \mathbb{N}}\in \mathcal{S}$ such that $\tau_n \downarrow v$, one has $E^\p[Y(\tau)] = \lim\limits_{n \rightarrow \infty} E^\p[Y(\tau_n)]$ with respect to each measure $\p \in \mathcal{P}.$ 
   \end{definition}
   
   \begin{proposition}($\pp$-RCE property for $R^+$) Let $(Y(\tau), \tau \in \mathcal{S})$ be an admissible family. The associated strict value function family $R^+ = (R^+(\tau), \tau \in \mathcal{S})$ is $\pp$-RCE.
   \end{proposition}
   
   \begin{proof}
   It is easily seen that for each $\p \in \mathcal{P}$, the function $\tau \rightarrow E^\p[R^+(\tau)]$ is a non-increasing function of stopping times, since $R^+$ is a $\mathcal{P}$-supermartingal family. Suppose, contrary to our claim, that there exists a probability measure $\bar{\p} \in \mathcal{P}$ such that the family $(R^+(\tau), \tau \in \mathcal{S})$ is not RCE at $\tau \in \mathcal{S}$. Consider first the case when $E^{\bar{\p}} [R^+(\tau)] <\infty$. Hence there exists a constant $\epsilon>0$ and a sequence of stopping times $(\tau_n)_{n\in \mathbb{N}}\in \mathcal{S}$ such that $\tau_n \downarrow \tau$  and 
  \begin{equation*}
  \lim\limits_{n \rightarrow \infty} \uparrow E^{\bar{\p}} [R^+(\tau_n)] + \epsilon \leq E^{\bar{\p}}[R^+(\tau)].
  \end{equation*}
  
  From Proposition \ref{expectation of R}, we have $E^{\bar{\p}}[R^+(\tau)]= \sup\limits_{\sigma \in \mathcal{S}_{\tau^+}} E^{\bar{\p}}[Y(\sigma)]$, thus, there exists $\tau' \in \mathcal{S}_{\tau^+} $ such that 
\begin{align} \label{tau'}
\lim\limits_{n \rightarrow \infty} \uparrow E^{\bar{\p}} [R^+(\tau_n)] + \frac{\epsilon}{2}  \leq E^{\bar{\p}}[Y(\tau')].
\end{align}
In order to find a contradiction, suppose at first that $\tau < T$ a.s. In this case, $\tau' \in \mathcal{S}_{\tau^+}$ gives $\tau' > \tau$ a.s. We write $ \lbrace \tau' > \tau \rbrace=  \bigcup\limits_{n\in\mathbb{N}} \uparrow \lbrace \tau' > \tau_n\rbrace $ and we get $E^{\bar{\p}}[Y(\tau')]= \lim\limits_{n \rightarrow \infty} \uparrow E^{\bar{\p}} [Y(\tau') \mathbf{1}_{\tau' >\tau_n}]$.
Then, there exists $n_0$ such that 
 \begin{align*}
 \lim\limits_{n \rightarrow \infty} \uparrow E^{\bar{\p}} [R^+(\tau_n)] + \frac{\epsilon}{4} \leq E^{\bar{\p}}[Y(\tau') \mathbf{1}_{\tau' >\tau_{n_0}}].
 \end{align*}
 
 Set $\bar{\tau}:= \tau' \mathbf{1}_{\tau' >\tau_{n_0}} + T \mathbf{1}_{\tau' \leq \tau_{n_0}}$. One can see that $\bar{\tau} > \tau_{n_0}$ a.s. Therefore, the positivity of $Y$ yields 
\begin{align} \label{es} E^{\bar{\p}}[R^+(\tau_{n_0})] + \frac{\epsilon}{4} \leq \lim\limits_{n \rightarrow \infty} \uparrow E^{\bar{\p}} [R^+(\tau_n)] + \frac{\epsilon}{4} \leq E^{\bar{\p}}[Y(\bar{\tau})] \leq E^{\bar{\p}}[R^+(\tau_{n_0})],\end{align} 
  which is impossible. To study the general case, take $ \tau \in \mathcal{S}$. Since $\tau' \in \mathcal{S}_{\tau^+}$, $\tau' > \tau$ a.s. on $\lbrace  \tau < T\rbrace$ and $\tau' =T$ a.s. on $\lbrace  \tau = T\rbrace$. Then, one has  
\begin{align*}
&E^{\bar{\p}}[Y(\tau')]=E^{\bar{\p}}[Y(\tau') \mathbf{1}_{\tau < T}] + E^{\bar{\p}}[Y(T) \mathbf{1}_{\tau = T}]\quad \text{and}\\
&E^{\bar{\p}}[Y(\tau') \mathbf{1}_{\tau <T}]= \lim\limits_{n \rightarrow \infty} \uparrow E^{\bar{\p}}[Y(\tau') \mathbf{1}_{\lbrace \tau <T \rbrace \cap \lbrace \tau' > \tau_n \rbrace  }].
\end{align*} 
From this and \eqref{tau'} there exists $n_0$ such that
 \begin{align*}
 \lim\limits_{n \rightarrow \infty} \uparrow E^{\bar{\p}} [R^+(\tau_n)] + \frac{\epsilon}{4} \leq E^{\bar{\p}}[Y(\tau') \mathbf{1}_{\tau' >\tau_{n_0} \cap \lbrace \tau < T \rbrace}] + E^{\bar{\p}}[Y(T) \mathbf{1}_{\tau = T}].
 \end{align*} 
 Consider $\bar{\tau}:= \tau' \mathbf{1}_{\lbrace \tau' > \tau_{n_0} \rbrace \cap \lbrace \tau < T \rbrace} + T \mathbf{1}_{\lbrace \tau' \leq \tau_{n_0} \rbrace  \cap \lbrace \tau < T \rbrace} + T \mathbf{1}_{\lbrace \tau = T \rbrace}$.  One can check that $\bar{\tau} \in \mathcal{S}_{\tau_{n_0}^+}$, then
 \begin{align*}
 E^{\bar{\p}}[Y(\tau') \mathbf{1}_{\lbrace \tau' >\tau_{n_0} \rbrace \cap \lbrace \tau < T \rbrace}] + E^{\bar{\p}}[Y(T) \mathbf{1}_{\lbrace \tau = T \rbrace}] \leq + E^{\bar{\p}}[Y(\bar{\tau})] \leq  E^{\bar{\p}}[R^+(\tau_{n_0})],
 \end{align*}
 we derive again \eqref{es} which gives a contradiction.\\
   \end{proof}
   
   The following definition plays a fundamental role in solving the optimal stopping problem \eqref{eq}.
   
    \begin{definition}
An admissible family $(Y(\tau), \tau \in \s)$ is said to be $\pp$-right (resp. $\pp$-left) upper-semicontinuous in expectation along stopping times if for all $\p \in \pp$, for all $\tau \in \s$ and for all sequences of stopping times $(\tau_n)_n$ such that $\tau_n \downarrow \tau$ (resp. $\tau_n \uparrow \tau$), we have
\begin{align*}
E^\p[Y(\tau)] \geq \lim\limits_{n \rightarrow  \infty}\sup E^\p[Y(\tau_n)].
\end{align*}
\end{definition}

\begin{remark} \label{USCE 1A}
Let $(Y(\tau), \tau \in \s)$ be a $\pp$-right (resp. $\pp$-left) USCE admissible family. For each $v \in \s$ and  $B \in \f_v$, the admissible family $(Y(\tau) \mathbf{1}_B, \tau \in \s_v)$ is also $\pp$-right (resp. $\pp$-left) USCE admissible family. Indeed, fix $\p \in \pp$ and $v\in \s$. Let $(\tau_n)_n$ be a sequence of stopping times such that $\tau_n \downarrow \tau$. For each $n$, set $\bar{\tau}_n:= \tau_n \mathbf{1}_B + T \mathbf{1}_{B^c}$ and $\bar{\tau}:= \tau \mathbf{1}_B + T \mathbf{1}_{B^c}$. It is easily seen that $\bar{\tau}_n \downarrow \tau$. Therefore, $E^\p[Y(\bar{\tau})] \geq \lim\limits_{n \rightarrow  \infty}\sup E^\p[Y(\bar{\tau}_n)].$
Hence, $E^\p[Y(\tau)\mathbf{1}_B] \geq \lim\limits_{n \rightarrow  \infty}\sup E^\p[Y(\tau_n)\mathbf{1}_B].$
\end{remark}
  
We give now necessary and sufficient conditions for optimality.

\begin{theorem}
Let $\p^* \in \mathcal{P}$ and $\tau^* \in\mathcal{S}$ be such that $E^{\p^*}[Y(\tau^*)]< \infty$. The stopping time $\tau^*$ and the probability measure $\p^*$ are optimal in \eqref{eq1*}, i.e., 
\begin{align} \label{opt}
E^{\p^*}[Y(\tau^*)]= R(0)= \sup_{\p \in \mathcal{P}} \sup_{\tau \in \mathcal{S}} E^\p[Y(\tau)]
\end{align}
holds, if and only if 
\begin{enumerate}
\item $R(\tau^*)= Y(\tau^*)$ a.s., \label{1}
\item \label{2} The $\mathcal{P}$-supermartingale family $(R(\tau^*  \wedge \sigma), \sigma \in \mathcal{S} )$ is a $\p^*$-martingale family, that is, for all $\theta, \theta' \in\mathcal{S}$ such that $\theta, \theta' \leq \tau^*$ a.s., we have 
\begin{align*}
    E^{\p^*}[R(\theta') \vert \mathcal{F}_\theta]= R(\theta) \quad \text{a.s. on} \quad \lbrace \theta \leq \theta'\rbrace.
\end{align*}
\end{enumerate}
\end{theorem}

\begin{proof}
Suppose $\tau^*$ and $\p^*$ are optimal, i.e. \eqref{opt} holds. From Proposition \ref{expectation of R}, we have  
\begin{align*} 
E^{\p^*}[R(\tau^*)]&= \sup_{\sigma \in \mathcal{S}_\tau^*} E^{\p^*}[Y(\sigma)] \leq \sup_{\p\in \mathcal{P}} \sup_{\sigma \in \mathcal{S}_\tau^*} E^{\p}[Y(\sigma)]\\
&\leq \sup_{\p\in \mathcal{P}} \sup_{\sigma \in \mathcal{S}} E^{\p}[Y(\sigma)] = E^{\p^*}[Y(\tau^*)]\\
&\leq E^{\p^*}[R(\tau^*)] \quad \text{a.s.}
\end{align*}
Since $R$ dominates $Y$ and  $E^{\p^*}[Y(\tau^*)]< \infty$, we get $R(\tau^*)= Y(\tau^*)$ $\p^*$-a.s. and the first assertion follows. 
Let us prove the second assertion. For this, take $\sigma \in \mathcal{S}$ and notice that 
\begin{align} 
E^{\p^*}[Y(\tau^*)]= \sup_{\p \in \mathcal{P}} \sup_{\tau \in \mathcal{S}} E^\p[Y(\tau)] = \sup_{\p \in \mathcal{P}} \sup_{\tau \in \mathcal{S}_{\sigma \wedge \tau^*}} E^\p[Y(\tau)]. \label{eq3}
\end{align}
Moreover, from Proposition \ref{expectation of R} we have 
\begin{align*}
\sup_{\p\in \mathcal{P}} E^{\p}[R(\sigma \wedge \tau^*)]= \sup_{\p \in \mathcal{P}} \sup_{\tau \in \mathcal{S}_{\sigma \wedge \tau^*}} E^\p[Y(\tau)] =  E^{\p^*}[Y(\tau^*)].
\end{align*}
From this and the supermartingale property of $R$, we get 
\begin{align*}
E^{\p}[R(\sigma \wedge \tau^*)] \leq  \sup_{\p \in \mathcal{P}} E^{\p}[R(\sigma \wedge \tau^*)]=E^{\p^*}[Y(\tau^*)] \leq  E^{\p^*}[R(\tau^*)]  \leq  E^{\p^*}[R(\tau^* \wedge \sigma)].
\end{align*}
It follows that, $E^{\p}[R(\sigma \wedge \tau^*)] = E^{\p^*}[R(\tau^*)]$, which gives the second assertion.
Conversely, from \ref{1} and \ref{2} we have
\begin{align*}
E^{\p^*}[R(\tau^*)]= E^{\p^*}[R(0)]= R(0) = E^{\p^*}[Y(\tau^*)]= \sup_{\p \in \mathcal{P}} \sup_{\tau \in \mathcal{S}} E^\p[Y(\tau)],
\end{align*}
the proof is thus complete. 

\end{proof}

%

\section{Existence of Optimal Stopping Times} \label{s4}

  In this section, we establish the existence of optimal stopping times using an approximation method introduced by Maingueneau \cite{Maingueneau}  (see also El Karoui \cite{El Karoui}). We begin by constructing a family of \textit{approximatly optimal} stopping times. Fix $v\in \s$. For $\alpha \in (0,1)$, a stopping time $\tau_\alpha^*$ is said to be  $(1-\alpha)$-optimal for $R(v)$ if it satisfies,
\begin{equation} \label{1-alpha}
\alpha R(v) \leq \esssup_{\p \in \pp} E^\p [ Y(\tau_\alpha^*) \vert
\f_v].
\end{equation}
Passing to the limit later when $\alpha \uparrow 1$ implies the existence of an optimal stopping time. Let us now define for $\alpha \in (0,1)$, the following $\f_v$-measurable random variable
\begin{align*} \label{U}
\U &:= \essinf \lbrace \tau \in \s_v, \alpha R(\tau) \leq Y(\tau) \quad \text{a.s.} \rbrace. \numberthis  
\end{align*}
\begin{remark} \label{rema}
\begin{enumerate}
    \item \label{L^alpha} The random variable $\U$ is a stopping time, and one has $\U \geq v$ a.s. Indeed, set $\lo := \lbrace \tau \in \s_v, \alpha R(\tau) \leq Y(\tau) \quad \text{a.s.}\rbrace$. Let $\tau_1$, $\tau_2 \in \lo$ and set $\bar{\tau}:= \tau_1  \vee \tau_2 \in \s_v$. We have
\begin{align*}
\alpha R(\bar{\tau})&= \alpha R(\tau_1) \mathbf{1}_{\lbrace \tau_1 \geq \tau_2 \rbrace} +  \alpha R(\tau_2) \mathbf{1}_{\lbrace \tau_1 < \tau_2 \rbrace} \\
& \leq Y(\tau_1) \mathbf{1}_{\lbrace \tau_1 \geq \tau_2 \rbrace} +  Y(\tau_2) \mathbf{1}_{\lbrace \tau_1 < \tau_2 \rbrace} \quad \text{a.s.}\\
&= Y(\bar{\tau}) \quad \text{a.s.}
\end{align*}
Thus, $\lo$ is stable by pairwise minimization. From \cite{Neveu}, there exists a sequence of stopping times $(\tau_n)_n \in \lo$ such that $\tau_n \downarrow \U$. Therefore, $\U$ is a stopping time and $\U \geq v$ a.s.
\item \label{U=v} Let $\alpha \in (0,1)$. Let $v\in \s$ such that $\alpha R(v) \leq Y(v)$ a.s. We have $\U = v$ a.s. Indeed, set $\bar{A}:= \lbrace \alpha R(v) \leq Y(v) \rbrace$ and $\bar{v}:= v \mathbf{1}_{\bar{A}} + T  \mathbf{1}_{ \bar{A}^c}$. Clearly, $\bar{v} \in \lo$. By definition of $\U$, we have $\U \leq \bar{v}$ a.s. Therefore,  $\U \mathbf{1}_{\bar{A}} \leq \bar{v}\mathbf{1}_{\bar{A}} = v \mathbf{1}_{\bar{A}}$ a.s., and hence  $\U = v$ a.s.
\end{enumerate}
\end{remark}

The following proposition holds.

\begin{proposition} \label{prop4.1}
Assume that $R<\infty$ and that the reward family $(Y(\tau), \tau \in \s)$ is $\pp$-right USCE. Then, for each $v \in \s$ and $\alpha \in (0,1)$, the stopping time $\U$, defined by \eqref{U}, satisfies  
\begin{align} \label{alphaR<Y}
\alpha R(\U) \leq Y(\U) \quad \text{a.s.}
\end{align}
\end{proposition}

\begin{proof}
Let $v \in \s$ and $A \in \f_{\U}$. Recall that, by Proposition \ref{R=YvR+}, on the set $B_\alpha:=\lbrace R(\U) > Y(\U) \rbrace$ we have $R(\U)=R^+(\U)$ a.s. Using the $\pp$-RCE property of $R^+$, it is easily seen that the family $(R^+(\tau) \mathbf{1}_{B_\alpha}, \tau \in \s_{\U})$ is also $\pp$-RCE. From Remark \ref{L^alpha}, there exists a minimizing sequence $(\tau_n)_n \in \lo$ such that $\tau_n \downarrow \U$. Hence, for each $\p \in \pp$, the following equality holds
\begin{align*}
\alpha E^\p[R^+(\U) \mathbf{1}_{B_\alpha \cap A} ] = \alpha  \lim\limits_{n \rightarrow \infty} E^\p[R^+(\tau_n) \mathbf{1}_{B_\alpha \cap A}].
\end{align*}
Since $(\tau_n)_n \in \lo$ and $R^+ \leq R$, we have for each $n$; $\alpha R^+(\tau_n) \leq \alpha R(\tau_n) \leq Y(\tau_n)$ a.s. Therefore, 
\begin{align} \label{eqlim}
\alpha E^\p[R^+(\U) \mathbf{1}_{B_\alpha \cap A }] \leq \alpha  \limsup\limits_{n \rightarrow \infty} E^\p[Y(\tau_n) \mathbf{1}_{B_\alpha \cap A}].
\end{align}
Hence, from the positivity of $Y$ and Eq. \eqref{eqlim}, we get
\begin{align*}
\alpha E^\p[R(\U) \mathbf{1}_A]&= \alpha E^\p[R^+(\U) \mathbf{1}_{B_\alpha \cap A }] + \alpha E^\p[Y(\U) \mathbf{1}_{B^c_\alpha \cap A}] \\
&\leq \limsup\limits_{n \rightarrow \infty} E^\p[Y(\tau_n) \mathbf{1}_{B_\alpha \cap A}] + \alpha E^\p[Y(\U) \mathbf{1}_{B^c_\alpha \cap A}] \\
&\leq  \limsup\limits_{n \rightarrow \infty} E^\p[Y(\tilde{\tau}_n) \mathbf{1}_{A}],
\end{align*}
where $\tilde{\tau}_n:= \tau_n \mathbf{1}_{B_\alpha \cap A} + \U \mathbf{1}_{B^c_\alpha \cap A} + \U \mathbf{1}_{A^c}$. The sequence $(\tilde{\tau}_n)_n$ is a non-increasing sequence of stopping times that tends to $\U$ as $n \rightarrow \infty$. Thus, using the $\pp$-right USCE assumption on the reward family $Y$ and Remark \ref{USCE 1A}, we obtain 
\begin{align*}
\alpha E^\p[R(\U) \mathbf{1}_A] \leq E^\p[Y(\U) \mathbf{1}_A],
\end{align*}
for each $A\in \f_{\U}$ and each $\p \in \pp$. Consequently, $\alpha R(\U) \leq Y(\U)$ a.s. and the proof is ended.\\
\end{proof}

Let $v\in \s$ and $\alpha \in (0,1)$. In the sequel, we show that $\U$ defined by \eqref{U} is a $(1-\alpha)$-optimal stopping time for $R(v)$. More precisely,

\begin{theorem}
Assume that the reward $(Y(\tau), \tau \in \s)$ is $\pp$-right USCE and $R< \infty$. Let $v \in \s$. For each $\alpha \in (0,1)$, the stopping time $\U$ is a $(1-\alpha)$-optimal stopping time for $R(v)$.
\end{theorem}
\begin{remark}
Note that the integrability condition $R<\infty$ is a weaker condition compared to the classical assumption made in the literature, namely, that $E^\p[\sup\limits_{\tau \in \s} Y(\tau)] < \infty$, $\forall \p \in \pp$ (see e.g. \cite{Zam}).
\end{remark}

\begin{proof}
We have to prove inequality \eqref{1-alpha} for $\U$. This will be done by means of the following steps. \\
\textit{Step 1.} Let us show that for each $\alpha \in (0,1)$ and each $v\in S$, we have 
\begin{align} \label{Eq4.5}
R(v)=  \esssup\limits_{\p \in \pp} E^\p [ R(U^\alpha(v)) \vert \f_v] \quad \text{a.s.}
\end{align}
 From the $\pp$-supermartingale property of $(R(v), v\in \s)$, we clearly have the property: for each $\p \in \pp$, $E^\p [ R(U^\alpha(v)) \vert \f_v]  \leq R(v)$ a.s., since $\U$ is a stopping time greater than or equal to $v$. By taking the essential supremum over $\p \in \pp$, we get $ \esssup\limits_{\p \in \pp} E^\p [ R(U^\alpha(v)) \vert \f_v] $\\$\leq R(v)$ a.s.  To prove the reverse inequality, consider for each $v\in \s$, the random variable $\bar{S}(v)=  \esssup\limits_{\p \in \pp} E^\p [ R(U^\alpha(v)) \vert \f_v]$.\\

\textit{Step 1.1.} We claim that the family $(\bar{S}(v), v\in \s)$ is a $\pp$-supermartingale family i.e., for any $\p \in \pp$ and $\tau, \tau' \in \s$ such that $\tau \geq \tau'$ a.s., we have 
 \begin{align} \label{supprop}
     E^\p[\bar{S}(\tau) \vert \f_{\tau'}] \leq \bar{S}(\tau') \quad \text{a.s.} 
\end{align} 
   Indeed, fix $\tau \in \s$, $Z\in \z$. Let $Z^1, Z^2\in \z$, and set
   \begin{align*}
       A&:= \Big \lbrace E^\q \Big[\frac{Z^1_{\Uu}}{Z^1_\tau} R(\Uu) \Big\vert \f_\tau\Big] \leq E^\q \Big[\frac{Z^2_{\Uu}}{Z^2_\tau} R(\Uu) \Big\vert \f_\tau \Big] \Big \rbrace \in \f_\tau.\\
       \bar{Z}_t&:= Z^1_t \q(A^c \vert \mathcal{F}_t) + Z^2_t \q(A \vert \mathcal{F}_t), \quad 0 \leq t \leq T. 
   \end{align*}
   Since $\mathcal{Z}$ is $L^0$-convex, we have $\bar{Z} \in \mathcal{Z}$. Therefore,
   \begin{align*} 
&\hspace*{0.5 cm} E^\q \Big [\frac{\bar{Z}_{\Uu}}{\bar{Z}_{\tau}} R(\Uu) \Big\vert \f_\tau \Big ] \\
 &= E^\q \Big[ \frac{Z^1_{\Uu}}{Z^1_\tau} R(\Uu) \Big\vert \mathcal{F}_\tau \Big ] \mathbf{1}_{A^c} + E^\q \Big [ \frac{Z^2_{\Uu}}{Z^2_\tau} R(\Uu) \Big\vert \mathcal{F}_\tau \Big] \mathbf{1}_{A}\\
 &= E^\q\Big[ \frac{Z^1_{\Uu}}{Z^1_\tau} R(\Uu) \Big\vert \mathcal{F}_\tau \Big]  \vee E^\q\Big[ \frac{Z^2_{\Uu}}{Z^2_\tau} R(\Uu) \Big\vert \mathcal{F}_\tau \Big].
 \end{align*}
 Hence, the family of random variables $( E^\q \big[\frac{Z_{\Uu}}{Z_\tau} R(\Uu)\vert \f_\tau \big], Z \in \z)$ is closed under pairwise maximization. Similarly as in proof of Proposition \ref{prop5}, without any loss of generality, there exists a sequence $(Z^n)_{n \in \mathbb{N}} \in \z_{\tau,\Uu}$, such that $Z^n_u = Z_u$, $\forall u \in [v,\tau]$ and
   \begin{align*}
\bar{S}(\tau) = \lim\limits_{n \rightarrow \infty} \uparrow E^\q \Big[\frac{Z^n_{\Uu}}{Z^n_\tau} R(\Uu) \Big\vert \f_\tau \Big] \quad \text{a.s.}
\end{align*}
It follows that, for any $\tau' \in \s_{\tau}$ a.s., we have 
\begin{align*}
       E^\p[\bar{S}(\tau) \vert \f_{\tau'}]&=E^\q \Big[ \frac{Z_\tau}{Z_{\tau'}} \bar{S}(\tau) \Big\vert \f_{\tau'} \Big] \\
       &= E^\q \Big[ \frac{Z_\tau}{Z_{\tau'}}  \lim\limits_{n \rightarrow \infty} \uparrow E^\q \Big[\frac{Z^n_{\Uu}}{Z^n_\tau} R(\Uu) \Big\vert \f_\tau \Big] \Big\vert \f_{\tau'} \Big] \\
       &\leq \lim\limits_{n \rightarrow \infty} E^\q \Big[ \frac{Z_\tau}{Z_{\tau'}} \frac{Z^n_{\Uu}}{Z^n_\tau} R(\Uu) \Big\vert \f_{\tau'} \Big] \\
       &\leq \lim\limits_{n \rightarrow \infty} E^\q \Big[ \frac{Z^n_{\Uu}}{Z^n_{\tau'}} R(\Uu) \Big\vert \f_{\tau'} \Big]\\
       &\leq \esssup_{M\in \z} E^\q \Big[ \frac{M_{\Uu}}{M_{\tau'}} R(\Uu) \Big\vert \f_{\tau'} \Big]\\
       &= \esssup_{M\in \z} E^\q \Big[ \frac{M_{U^\alpha(\tau')}}{M_{\tau'}}  E^\q \Big[ 
      \frac{M_{\Uu}}{M_{U^\alpha(\tau')}} R(\Uu) \Big\vert \f_{U^\alpha(\tau')} \Big] \Big\vert \f_{\tau'} \Big].  \numberthis \label{eq46}
\end{align*}
 Since $\tau \geq \tau'$, it is immediate that $\Uu \geq U^\alpha(\tau')$. The $\pp$-supermatingale property of $R$ together with Eq. \eqref{eq46} yields
 \begin{align*}
     E^\p[\bar{S}(\tau) \vert \f_{\tau'}] \leq \esssup_{M\in \z} E^\q \Big[ \frac{M_{U^\alpha(\tau')}}{M_{\tau'}}  R(U^\alpha(\tau')) \Big\vert \f_{\tau'} \Big] = \bar{S}(\tau') \quad \text{a.s.,}
 \end{align*}
 which is our claim in \eqref{supprop}. \\
 
\textit{Step 1.2.} For a fixed $\alpha \in (0,1)$, consider the $\pp$-supermartingale family $(\alpha R(v) + (1- \alpha) \bar{S}(v), v \in \s)$. Let us show that $\alpha R(v) + (1- \alpha) \bar{S}(v) \geq Y(v)$. Fix $\tau \in \s_v$. Set $\bar{A}:= \lbrace \alpha R(v) \leq Y(v) \rbrace \in \f_v$. From \ref{U=v} of Remark \ref{rema}, we have $\U= v$ a.s. on $\bar{A}$. Hence, $\bar{S}(v)=  \esssup\limits_{\p \in \pp} E^\p [ R(U^\alpha(v)) \vert \f_v]=\esssup\limits_{\p \in \pp} E^\p [ R(v) \vert \f_v]= R(v)$ a.s. on $\bar{A}$. Therefore,
\begin{align*}
    \alpha R(v) + (1- \alpha) \bar{S}(v) = R(v) \geq Y(v) \quad \text{a.s. on $\bar{A}$.}
\end{align*}
Moreover, on $\bar{A}^c= \lbrace \alpha R(v) > Y(v) \rbrace$ by the positivity of $\bar{S}$, we get 
\begin{align*}
    \alpha R(v) + (1- \alpha) \bar{S}(v) \geq \alpha R(v) \geq Y(v) \quad \text{a.s. on $\bar{A}^c$,}
\end{align*}
and Step 1.2. is proved. \\
Now since $(R(v), v \in \s)$ is the smallest $\pp$-supermartingale family which dominates $(Y(v), v\in \s)$, we get $ \alpha R(v) + (1- \alpha) \bar{S}(v) \geq R(v)$ a.s. Moreover, since  $R(v) < \infty$ a.s. and $\lambda < 1$, we conclude that $\bar{S}(v) \geq R(v)$ a.s. The proof of Step 1 is complete. \\

\textit{Step 2.} Let us now show that 
\begin{equation}
\alpha R(v) \leq \esssup_{\p \in \pp} E^\p [ Y(\U) \vert
\f_v] \quad \text{a.s.}
\end{equation}
By Proposition \ref{prop4.1} and Step 1., we have
\begin{align*}
    \alpha R(v)= \esssup\limits_{\p \in \pp} E^\p [ \alpha R(U^\alpha(v)) \vert \f_v] \leq \esssup\limits_{\p \in \pp} E^\p [ Y(U^\alpha(v)) \vert \f_v] \quad \text{a.s.}
\end{align*}
Thereupon, $\U$ is $(1-\alpha)$-optimal for $R(v)$, and the proof is complete. \\
\end{proof}

Let $v\in \s$ and $\lambda_1, \lambda_2 \in (0,1)$ such that $\lambda_1 \leq \lambda_2$, then clearly $U^{\lambda_1}(v) \leq U^{\lambda_2}(v)$ a.s. Accordingly, the map $\alpha \rightarrow \U$ is non-decreasing on $(0,1)$, and so we define the stopping time 
\begin{align} \label{U*}
    U^*(v):= \lim\limits_{\lambda \uparrow 1} \U \quad \text{a.s.}
\end{align}
Note that $U^*(v) \geq v$ a.s. Inspiring by classical literature, the stopping time $U^*(v)$ appears to be a good nominee for optimal stopping time for $R(v)$. Under further assumptions on the reward family, we state the main existence result of an optimal stopping time in the following theorem. 
\begin{theorem}(Existence of an optimal stopping time) \label{existence of optimal stopping time}
Assume that the reward family \\ $(Y(v),v \in \s)$  is $\pp$-USCE, and that $R<\infty$. Then for every $v \in \s$, the stopping time $U^*(v)$ (defined by \eqref{U*}) is an optimal stopping time for $R(v)$, that is, \begin{align} \label{opt for v}
    R(v)= \esssup_{ \p \in \pp}  E^\p[Y(U^*(v)) \vert \f_v].
\end{align}
Additionally, $U^*(v)=\mathcal{U}(v):= \essinf \lbrace \tau \in\s_v, R(\tau)=Y(\tau) \quad \text{a.s.} \rbrace$ a.s. 
\end{theorem}

\begin{proof} Let $v\in \s$. By the characterization of  $(R(v), v\in \s)$ as the smallest $\pp$-supermartingale family which dominates $(Y(v), v\in \s)$, for every $\p \in \pp$, we have
\begin{align*}
    E^\p[Y(U^*(v)) \vert \f_v] \leq E^\p[R(U^*(v)) \vert \f_v] \leq R(v) \quad \text{a.s.,}
\end{align*}
which yields that $\esssup\limits_{\p \in \pp} E^\p[Y(U^*(v)) \vert \f_v] \leq R(v)$ a.s. Let us show the other inequality. Suppose by contradiction that for all $\epsilon>0$, $\p \in \pp$,
\begin{align*}
    \p\big( \lbrace R(v) - \epsilon \geq E^\p[Y(U^*(v))\vert \f_v] \rbrace \big) >0. 
\end{align*}
Set $C:=\lbrace R(v) - \epsilon \geq E^\p[Y(U^*(v))\vert \f_v] \rbrace$. For all $A\in \f_v$, we have
\begin{align*}
    E^\p[Y(U^*(v)) \mathbf{1}_A]&= E^\p[E^\p[Y(U^*(v))\vert \f_v] \mathbf{1}_A]\\
    &= E^\p \Big[E^\p[Y(U^*(v))\vert \f_v] \mathbf{1}_{A \cap C} \Big] + E^\p[Y(U^*(v)) \mathbf{1}_{A \cap C^c}]\\
    &\leq E^\p [(R(v) - \epsilon) \mathbf{1}_{A \cap C} ] + E^\p[R(U^*(v)) \mathbf{1}_{A \cap C^c}]  \quad \text{a.s.} \\
    &\leq E^\p [(R(v) - \epsilon) \mathbf{1}_{A \cap C} ] + E^\p[R(v) \mathbf{1}_{A \cap C^c}]  \quad \text{a.s.}\\
    &\leq E^\p[(R(v)- \epsilon \mathbf{1}_C) \mathbf{1}_A] \quad \text{a.s.}
\end{align*}
Therefore, $E^\p[Y(U^*(v)) \vert  \f_v] \leq R(v)- \epsilon \mathbf{1}_C$ a.s. By the integrability condition $R<\infty$, Proposition \ref{alphaR<Y} and Remark \ref{USCE 1A}, we have for each $A\in \f_v$
\begin{align*}
    E^\p[\lim\limits_{\alpha \uparrow 1} E^\p[R(\U) \vert \f_v] \mathbf{1}_A ] &\leq  E^\p \Big[\lim\limits_{\alpha \uparrow 1} E^\p \Big[ \frac{1}{\alpha} Y(\U) \mathbf{1}_A  \vert \f_v \Big] \Big]\\
    &\leq  \lim\limits_{\alpha \uparrow 1} E^\p \Big[ \Big (\frac{1}{\alpha} Y(\U) \mathbf{1}_A \Big) \Big] \\
    & \leq E^\p [Y(U^*(v)) \mathbf{1}_A].
\end{align*}
Hence, 
\begin{align*}
    \lim\limits_{\alpha \uparrow 1} E^\p[R(\U) \vert \f_v] \leq  E^\p [Y(U^*(v)) \vert \f_v] \leq R(v) - \epsilon \mathbf{1}_C, \quad \forall \p \in \pp. 
\end{align*}
Then, there exists $\bar{\alpha} \in (0,1)$ such that, for each $\p \in \pp$,
\begin{align*}
    E^\p[R(U^{\bar{\alpha}} (v)) \vert \f_v] \leq  E^\p [Y(U^*(v)) \vert \f_v] \leq R(v) - \epsilon \mathbf{1}_C.
\end{align*}
Taking the essential supremum over $\p \in \pp$, and using Eq. \eqref{Eq4.5}, we get 
\begin{align*}
    R(v) \leq R(v) - \epsilon \mathbf{1}_C,
\end{align*}
which gives the desired contradiction, and hence Eq. \eqref{opt for v} follows. Now, what is left is to show that $\mathcal{U}(v)= U^*(v)$ a.s. Since the map $\alpha \rightarrow \U$ is non-decreasing on $(0,1)$, for each $\alpha \in (0,1)$, we have $\U \leq U^*(v)$ a.s. By Proposition \ref{prop4.1} and the $\pp$-supermartingale property of $(R(v), v \in \s)$, we obtain
\begin{align*}
    \alpha E^\p[R(U^*(v))] \leq \alpha E^\p[R(\U)] \leq E^\p[Y(\U)] \quad \text{a.s.,} \quad \forall \alpha \in (0,1), \quad \forall \p \in \pp.
\end{align*}
Passing to the limit as $\alpha \uparrow 1$, and using the fact that $(Y(v), v\in \s)$ is $\pp$-USCE, we get
\begin{align*}
      E^\p[R(U^*(v))] \leq \liminf\limits_{\alpha \uparrow 1} E^\p[R(\U)] \leq \lim\limits_{\alpha \uparrow 1} E^\p[Y(\U)] \leq E^\p[Y(U^*(v))] \quad \text{a.s.,} 
\end{align*}
for every $\p \in \pp$. This with the fact that the family $(R(v), v \in \s)$ dominates $Y$, leads to  $R(U^*(v))=Y(U^*(v))$ a.s. It follows that, by definition of $\mathcal{U}(v)$, we have $U^*(v) \geq \mathcal{U}(v)$. Let us show the other inequality. Observe that for $\alpha=1$, we have $U^1(v)= \mathcal{U}(v)$ a.s. Then, for all $\alpha \leq 1$, $ \U \leq U^1(v)= \mathcal{U}(v)$ a.s. Taking the limit $\alpha \uparrow 1$, we obtain $U^*(v) \leq \mathcal{U}(v)$ a.s., which makes the proof ended. 
\end{proof}

\section{Existence of Optimal models} \label{optimmod}

We now consider the question of under which conditions on the family $\mathcal{P}$ there exists an optimal probability model for our problem. To this end, under suitable conditions on the family $\mathcal{P}$, we establish a “universal” Doob-Meyer-Mertens's decomposition for our Snell envelope family $\mathcal{R}=(R(v), v\in\s)$ in the sense that it holds simultaneously for all $\p \in \pp$. Precisely, we shall decompose $\mathcal{R}$ as the difference between a $\pp$-martingale with RCLL paths, and an optional RCLL increasing process.\\ 

Recall some definitions and results. Here and subsequently,  
\begin{enumerate}
    \item[•] $\mathbf{H}^2$ (resp. $\mathbf{H}^{loc}$) denotes the set of all zero-mean, square-integrable $\q$-martingales (resp. local $\q$-martingales). 
    We define semi-norms on $\mathbf{H}^2$ by the formula: $\normx{M}_t:= \sqrt{E^{\q}[M_t^2]}$ for $M\in \mathbf{H}^2$. The space $\lbrace \mathbf{H}^2, \normx{.}_t \rbrace$ is a complete separable space (see \cite[Proposition 4.1.]{kunitawatanabe}).
    \item[•] $\mathbf{X}^2_{lr}$ denotes the collection of all optional processes $X$ such that $X_\tau$ lies in $L^2(\f_\tau,\q)$ for all $\tau\in \s$, and such that $X$ admits right and left-limits. 
    \item[•] For every $\q$-local martingale $H$, $\mathcal{E}(H)$ denotes the Doléans-Dade exponential of $H$, that is, the solution of the stochastic differential equation: $d H_t=H_{t-} d H_t$, with $H_0=1$.
\end{enumerate}

\begin{definition} \label{subspace}
A subset $\mathcal{H}$ of $\mathbf{H}^2$ is called a subspace of $\mathbf{H}^2$ if it satisfies the following three conditions; $(i)$ \quad  $X,Y \in \mathcal{H}$ then $X+Y \in \mathcal{H}$; \quad $(ii)$ \quad If $X\in \mathcal{H}$ and $\psi$ is predictable, then $\int \psi dX \in \mathcal{H}$; \quad  $(iii)$ \quad $\mathcal{H}$ is closed in $\lbrace \mathbf{H}^2, \normx{.}_t \rbrace$.
\end{definition}

We have the following fundamental “Kunita-Watanabe decomposition” (see \cite[Proposition 4.2.]{kunitawatanabe}. 

\begin{proposition}
    If $\mathcal{H}$ is a subspace of $\mathbf{H}^2$, then any element $W$ of $\mathbf{H}^2$ can be decomposed uniquely as $W=W'+W''$, where $W' \in \mathcal{H}$, and $W'' \in \mathcal{H}^\perp$.
\end{proposition}

Note that if two probability measures $\p$ and $\q$ are equivalent, then there is a $\q$-local martingale $H$, such that $\left.\frac{d \p}{d \q}\right|_{\mathcal{F}_t}=\mathcal{E}(H)_t$ (see a.g., \cite[Proposition 5.8.]{legall}). We denote by $\mathcal{H}$ the set
\begin{equation}
    \mathcal{H} \triangleq \lbrace H  \in \mathbf{H}^{loc} / \,\, \mathcal{E}(H)= Z, \quad \text{for some } Z \in \mathcal{Z} \rbrace. 
\end{equation}
We can now state the following universal Doob-Meyer-Mertens's decomposition that holds simultaneously for all $\p \in \pp$. The proof is adapted from the Kunita-Watanabe work on square-integrable martingales (see
\cite{kunitawatanabe,Zam}).

\begin{theorem}(Optional decomposition) \label{universal decompo}
Suppose that $\mathcal{H}$ is a subspace of $\mathbf{H}^2$ in the sense of Definition \ref{subspace}. Let $S$ be a  $\pp$-supermartingale family $S$, such that $(S(\tau), \tau \in \s)$ is uniformaly integrable w.r.t. the reference probability $\q$, then there exists $X\in \mathbf{X}^2_{lr}$ such that $S(\tau) = X_\tau$ a.s. for all $\tau\in \s$. If in addition, $\esssup\limits_{\tau \in \s} S(\tau) \in  L^2(\f,\q)$, then there exists $C$ an optional non-decreasing process with RCLL paths and $C_0=0$, and a RCLL $\pp$-martingale $M\in \mathbf{H}^2$, such that a.s.
\begin{equation}
    S(\tau)=X_\tau= X_0+ M_\tau - C_\tau, \quad \text{for all $\tau \in \s$}. 
\end{equation}
\end{theorem}

\begin{proof}
The existence of the process $X\in \mathbf{X}^2_{lr}$ such that $S(\tau) = X_\tau$ a.s. for all $\tau\in \s$ follows from \cite[Theorem 3.1]{g-sup systems decomposition}.
Using Theorem 3.1 in \cite{g-sup systems decomposition}, we deduce the existence of an increasing predictable process $A^\q$ such that $A^\q_0=0$ and $A^\q_T \in L^2(\f, \q)$, and of a $\q$-martingale $\tilde{M}^\q \in \mathbf{H}^2$ such that a.s.
\begin{equation*}
    S(\tau)=X_\tau= X_0+ \tilde{M}^\q_\tau - A^\q_\tau, \quad \forall \tau \in \s.
\end{equation*}
Since $\mathcal{H}$ is a subspace of $\mathbf{H}^2$, the martingale $\tilde{M}^\q$ admits the Kunita-Watanabe decomposition: $M^\q= K+ M$, where $K \in \mathcal{H}$ and $M \in \mathcal{H}^\perp$, i.e. $\langle H, M \rangle=0, \forall H \in \mathcal{H}$. By \cite[Remark 2.16]{Zam}, the set $\z$ of \eqref{Z} is also closed in the set of all square-intergable $\q$-martingales $Z$ with mean 1, under the semi-norms$\normx{.}_t$ defined by: $\normx{Z}_t:= E^{\q}[(Z_t-1)^2]^{\frac{1}{2}}$. Consequently, $ \langle \mathcal{E}(H),M \rangle=0$ for all $H \in \mathcal{H}$, and for each $\p \in \pp$, $Z^\p M$ is a $\q$-martingale. This implies that $M$ is a $\pp$-martingale, which we identify with the $\pp$-martingale in our decomposition. Therefore, we have a.s.
\begin{equation*}
    S(\tau)=X_\tau= X_0+ M_\tau -( A^\q_\tau - K_\tau), \quad \forall \tau \in \s.
\end{equation*}
The proof is completed by showing that the optional process $C:= A^\q -K$ is non-decreasing. Notice that the square-integrable martingale $K \in \mathcal{H}$ admits a decomposition of the form 
\begin{equation*}
    K= K^c + K^d,
\end{equation*}
in which $K^c$ its continuous martingale part, and $K^d$ its purely discontinuous martingale
part, such that $\langle K^c, K^d \rangle = 0$ (see e.g., \cite{Probabilites et Potentiel2}). If we prove that $K$ is a purely discontinuous martingale that has only negative jumps, the assertion follows. 
Indeed, since $S$ is a $\pp$-supermartingale family, $X$ is a $\pp$-supermartingale; thus due to the Girsanov Theorem, the process $\langle H, K+M \rangle - A^\q $ is non-increasing, for all $H \in \mathcal{H}$. It follows that 
\begin{equation} \label{increasing}
    \text{$A^\q - \langle H, K \rangle$ is an increasing process, $\forall H \in \mathcal{H}$.}
\end{equation}  
Now since $K \in \mathbf{H}^2$, the process $\langle K^c \rangle$ is integrable, and by the Lebesgue decomposition theorem, the measure $d A^\q_t$ admits a decomposition of the form 
\begin{equation} \label{dA_t}
    d A^\q_t=\psi_t d\left\langle K^c\right\rangle_t+d D_t,
\end{equation}
in which $\psi$ is a positive predictable process in $L^1([0, \infty) \times$ $\Omega, d\left\langle K^c\right\rangle d Q)$, and $D$ is an integrable predictable increasing process such that, $\q$-almost surely, $d D_t \perp d\left\langle K^c\right\rangle_t$. Fix $m\in \mathbb{Z}$. We have the following version of \eqref{dA_t}
\begin{equation} \label{dAp_t}
    d A^\q_t=\psi_t \mathbf{1}_{\lbrace \psi_t \leq m \rbrace } d\left\langle K^c\right\rangle_t + d  D^m_t,
\end{equation}
where, $\q$-almost surely, the measure $d D^m_t $ is singular w.r.t. $ \mathbf{1}_{\lbrace \psi_t \leq m \rbrace } d\left\langle K^c\right\rangle_t$. Thus, on ${\lbrace \psi \leq m \rbrace }$ we have 
\begin{equation} \label{AQ=}
    A^\q_t= \int_0^t \psi_s \mathbf{1}_{\lbrace \psi_s \leq m \rbrace } d\left\langle K^c\right\rangle_s.
\end{equation}
Since $\mathcal{H}$ is a subspace of $\mathbf{H}^2$, we have $ (\int_0^t  (1+ \psi_s) \mathbf{1}_{\lbrace \psi_s \leq m \rbrace } d K^c_s )_{t\in[0,T]} \in \mathcal{H}$. Therefore, on the event ${\lbrace \psi \leq m \rbrace }$, \eqref{increasing} shows that the process 
\begin{equation*}
    A^\q - \left\langle \int  (1+ \psi_s) \mathbf{1}_{\lbrace \psi_s \leq m \rbrace } d K^c_s, K \right\rangle =  A^\q - \int  (1+ \psi_s) \mathbf{1}_{\lbrace \psi_s \leq m \rbrace } d \left\langle K^c \right\rangle_s \,\,\, \text{is increasing.}
\end{equation*}
By \eqref{AQ=}, we conclude that $\q$-almost surely, the process 
\begin{equation*}
     \left\langle K^c \right\rangle_t = \int_0^t  (1+ \psi_s) \mathbf{1}_{\lbrace \psi_s \leq m \rbrace } d \left\langle K^c \right\rangle_s - \int_0^t \psi_s \mathbf{1}_{\lbrace \psi_s \leq m \rbrace } d\left\langle K^c\right\rangle_s
\end{equation*}
is non-increasing on $\lbrace \psi \leq m \rbrace$, $\forall m \in\mathbb{Z}$. Consequently, $\left\langle K^c \right\rangle_\infty = 0$, and so $K^c =0$ and $K=K^d$. We now proceed to show that $K$ has only negative jumps. The square-integrable martingale $K$ can be decomposed further with respect to the sign of its jumps:  $K= K^+ + K^-$, where $K^+$ (resp. $K^-$) is the compensated integral of $\mathbf{1}_{\lbrace \Delta K>0 \rbrace}$ (resp. $\mathbf{1}_{\lbrace \Delta K<0 \rbrace}$) with respect to $K$. The processes $K^+$ and $K^-$ are both square-integrable martingales (see \cite[Ch. VIII, p. 357]{Probabilites et Potentiel2}). We proceed analogously to the proof of $K^c=0$, and we show that $K^+=0$. We conclude that $K$ is a purely discontinuous martingale with negative jumps, hence the process $C_t= A^\q_t -K_t$ is non-decreasing, which is the desired conclusion.\\
\end{proof}

The following proposition shows that our Snell envelope family $\mathcal{R}=(R(v), v\in\s)$ satisfies the integrability condition of Theorem \ref{universal decompo}; thus, it admits the “universal” Doob-Meyer-Mertens's decomposition. 

\begin{proposition}
    \label{universal decompo for R propo}
If $\mathcal{H}$ is a subspace of $\mathbf{H}^2$ in the sense of Definition \ref{subspace}, and if
\begin{equation} \label{integ cond for Y}
    \sup\limits_{\p \in \pp} E^\p[(\esssup\limits_{\tau \in \s} Y(\tau))^2] < \infty,
\end{equation}
then the Snell envelope family $\mathcal{R}=(R(v), v\in\s)$ associated with $(Y(\tau), \tau \in \s)$ admits a “universal” Doob-Meyer-Mertens's decomposition; i.e.,
there exists $X\in \mathbf{X}^2_{lr}$ such that $R(\tau) = X_\tau$ a.s. Moreover, there exists $C$ an optional non-decreasing process with RCLL paths and $C_0=0$, and a RCLL $\pp$-martingale $M\in \mathbf{H}^2$, such that a.s.
\begin{equation} \label{decomp of R}
    R(\tau)=X_\tau= X_0+ M_\tau - C_\tau, \quad \text{for all $\tau \in \s$}. 
\end{equation}
\end{proposition}

\begin{proof}
Let us show that $\esssup\limits_{\tau \in \s} R(\tau) \in L^2(\f, \q)$. Actually, condition \eqref{integ cond for Y} gives even more; we shall prove that $E^\p[(\esssup\limits_{\tau \in \s} R(\tau))^2]<\infty$, for any probability measure $\p \in \pp$. Indeed, let $\bar{Y}$ denote the random variable $\esssup\limits_{\sigma \in \s} Y(\sigma)$. Notice that we have a.s.
\begin{align*}
    R(\tau)= \esssup\limits_{\p \in \pp} \esssup\limits_{\sigma \in \s_\tau} E^\p[ Y(\sigma) \vert \f_\tau] 
    \leq  \esssup\limits_{\p \in \pp}  E^\p[ \bar{Y} \vert \f_\tau] =: G(\tau), \quad \forall \tau \in \s. 
\end{align*}
 We first show that the family $\mathcal{G}=(G(\tau),\tau \in \s)$ is a square-integrable $\pp$-supermartingale family, with $G(T)=\bar{Y}$. It is easy to check that $\lbrace E^\p[\bar{Y} \vert \f_\tau], \p \in \pp \rbrace$, for each $\tau \in \s$, and  $\lbrace E^\p[\bar{Y}^2 \vert \f_v]/ v\in \s, \p \in \pp \rbrace$ are closed under pairwise maximisation. Fix an arbitrary $\bar{\p} \in \pp$. There exists a sequence $\lbrace \p_n \rbrace_n$ of probability measures in $\pp$, such that $G(\tau)=\lim\limits_{n \rightarrow \infty} \uparrow E^{\p_n}[\bar{Y} \vert \f_\tau]$ a.s. Without loss of generality, we can take $\p_n=\bar{P}$ on $\f_\tau$, $\forall n$. Hence, by monotone convergence we get, for each $\tau, v\in \s$ s.t. $T\geq \tau \geq v \geq 0$
 \begin{align*}
     E^{\bar{\p}}[G(\tau) \vert \f_v] &= E^{\bar{\p}} \Big [ \lim\limits_{n \rightarrow \infty} \uparrow E^{\p_n} [\bar{Y} \vert \f_\tau] {\Big \vert} \f_v \Big]
     = \lim\limits_{n \rightarrow \infty} \uparrow E^{\bar{\p}} \Big[ E^{\p_n}[\bar{Y} \vert \f_\tau] {\Big \vert} \f_v \Big] \\
     &= \lim\limits_{n \rightarrow \infty} \uparrow E^{\p_n} \Big [ E^{\p_n}[\bar{Y} \vert \f_\tau] {\Big \vert} \f_v \Big ] = \lim\limits_{n \rightarrow \infty} \uparrow E^{\p_n}[\bar{Y} \vert \f_v] \\
     &\leq \esssup\limits_{\p \in \pp}  E^\p[ \bar{Y} \vert \f_v] = G(v) \quad \text{a.s.,}
 \end{align*}
and thus the “supermartingale" property of $\mathcal{G}$ holds. In the same manner, we obtain the square-integrable property;
\begin{align*}
     E^{\bar{\p}}[G(\tau)^2] &= E^{\bar{\p}} \Big [ \lim\limits_{n \rightarrow \infty}  (E^{\p_n}[\bar{Y} \vert \f_\tau])^2 \Big]
     \leq E^{\bar{\p}} \Big [ \liminf\limits_{n \rightarrow \infty}  E^{\p_n}[(\bar{Y}^2) \vert \f_\tau] \Big] \\
     &\leq \liminf\limits_{n \rightarrow \infty}  E^{\bar{\p}} \Big [  E^{\p_n}[(\bar{Y})^2 \vert \f_\tau] \Big] = \liminf\limits_{n \rightarrow \infty}  E^{\p_n}  \Big [E^{\p_n}[(\bar{Y})^2 \vert \f_\tau] \Big]  \\
     &\leq \sup\limits_{\p \in \pp}  E^\p[ (\bar{Y})^2] < \infty,
 \end{align*}
using Fatou’s lemma, Jensen’s inequality, and the condition \eqref{integ cond for Y}. Then, let us consider the following $\bar{\p}$-supermartingale family 
\begin{equation}
    H(\tau):= G(\tau) - E^{\bar{\p}}[G(T) \vert \f_\tau], \quad \tau \in \s.
\end{equation}
Notice that for each $\tau\in \s$, $0\leq H(\tau) \leq G(\tau)$. Thus, we have   
\begin{align*}
    H(\tau)^2 \leq G(\tau)^2 \leq \liminf\limits_{n \rightarrow \infty}  E^{\p_n}[(\bar{Y})^2 \vert \f_\tau] \leq \esssup\limits_{\p \in \pp} E^{\p} [(\bar{Y})^2 \vert \f_\tau],
\end{align*}
and
\begin{align*}
     \esssup\limits_{\tau \in \s} H(\tau)^2 \leq  \esssup\limits_{\p \in \pp} \esssup\limits_{\tau \in \s} E^{\p} [(\bar{Y})^2 \vert \f_\tau] = \lim\limits_{n \rightarrow \infty} \uparrow E^{\q_n}[(\bar{Y})^2 \vert \f_{\tau_n}] \quad \text{a.s.}
\end{align*}
where $\lbrace (\tau_n, \q_n) \rbrace_{n \in \mathbb{N}}$ is a conveniently chosen sequence in $\s \times \pp$ such that $\q_n=\bar{\p}$ on $\f_{\tau_n}$, $\forall n \in \mathbb{N}$. 
We take the expectation with respect to $\bar{\p}$ and apply the monotone convergence theorem to obtain that 
\begin{align*}
    E^{\bar{\p}}[\esssup\limits_{\tau \in \s} H(\tau)^2] \leq \lim\limits_{n \rightarrow \infty} \uparrow  E^{\bar{\p}} \Big[ E^{\q_n}[(\bar{Y})^2 \vert \f_{\tau_n}] \Big] &\leq \lim\limits_{n \rightarrow \infty} \uparrow  E^{\q_n}[(\bar{Y})^2 ] \\ &\leq \sup\limits_{\p \in \pp} E^\p[(\bar{Y})^2] < \infty.
\end{align*}
Hence, the system $(H(\tau), \tau \in \s)$ is a square-integrable $\bar{\p}$-supermartingale family with the property: $\esssup\limits_{\tau \in \s} H(\tau) \in L^2(\f, \bar{\p})$. We then apply Theorem 3.1 in \cite{g-sup systems decomposition} to deduce the existence of a process $Z\in \mathbf{X}^2_{lr}$, an increasing predictable process $B$ such that $B_0=0$ and $B_T \in L^2(\f, \bar{\p})$, and of a $\bar{\p}$-martingale $N \in \mathbf{H}^2$ such that a.s.
\begin{equation}
    H(\tau)=Z_\tau= Z_0+ N_\tau - B_\tau, \quad \forall \tau \in \s.
\end{equation}
We check at once that $H(\tau)= G(\tau) - E^{\bar{\p}}[\bar{Y} \vert \f_\tau]=  E^{\bar{\p}}[B_T \vert \f_\tau]- B_\tau$, for each $\tau \in \s$ a.s., and hence 
\begin{align*}
    G(\tau) \leq E^{\bar{\p}}[\bar{Y} + B_T \vert \f_\tau], \quad \forall \tau \in \s \,\, \text{a.s.}
\end{align*}
Notice that the process $(E^{\bar{\p}}[\bar{Y} + B_T \vert \f_t])_{t\in [0,T]}$ is a RCLL square-integrable $\bar{\p}$-martingale, since $B_T \in L^2(\f, \bar{\p})$ and \eqref{integ cond for Y} holds. Therefore, by the Burkholder–Davis–Gundy inequalities, we deduce that 
\begin{align*}
    E^{\bar{\p}} \Big[(\esssup\limits_{\tau \in \s} R(\tau))^2 \Big] &\leq E^{\bar{\p}} \Big[(\esssup\limits_{\tau \in \s} E^{\bar{\p}}[\bar{Y} + B_T \vert \f_\tau])^2 \Big] \\
    &=  E^{\bar{\p}} \Big[(\sup\limits_{0 \leq t \leq T} E^{\bar{\p}}[\bar{Y} + B_T \vert \f_t])^2 \Big] < \infty, 
\end{align*}
for every $\p \in \pp$. Hence, we can use Theorem \ref{universal decompo}, which provides the desired decomposition. The proof is complete. \\
\end{proof}

We are now in position to state the main result of this section. The object is to characterize an optimal probability model using the above decomposition \eqref{decomp of R}. This is our Theorem \ref{existence of optimal model} below.
 
\begin{theorem} \label{existence of optimal model}
Assume that the set $\mathcal{H}$ is a subspace of  $\mathbf{H}^2$, that the reward family $(Y(v),v \in \s)$  is $\pp$-USCE, and that equation \eqref{integ cond for Y} holds. Then there exists a probability measure $\p^* \in \pp$ such that, for every $v \in \s$,
\begin{equation} 
E^{\p^*} \Big [Y(U^*(v)) \Big\vert \f_v \Big ]= R(v)=  \esssup_{ \p \in \pp}  E^\p \Big[Y(U^*(v)) \Big\vert \f_v \Big] \quad \text{a.s.}
 \end{equation}
Moreover, any model $\p \in \pp$ is then optimal. 
\end{theorem}

\begin{proof}
Let $\p \in \pp$ and $v \in \s$. The conditions of Theorem \ref{existence of optimal stopping time} and Theorem \ref{universal decompo for R propo} are  fulfilled, thus
\begin{align*}
    E^\p  [R(U^*(v)) \vert \f_v  ]= E^\p  [Y(U^*(v)) \vert \f_v  ] = R(v) - E^\p [C_{U^*(v)} - C_v \vert \f_v], \quad \text{a.s.}
\end{align*}
Taking the essential supremum with respect to $\p$ yields
 \begin{equation}
     \essinf_{ \p \in \pp}  E^\p \Big [C_{U^*(v)} - C_v \Big\vert \f_v \Big] = \essinf_{ Z \in \z}  E^\q  \Big [\frac{Z_{U^*(v)}}{Z_v} (C_{U^*(v)} - C_v) \Big\vert \f_v \Big]  = 0, \quad \text{a.s.}
 \end{equation}
 The set $\lbrace E^\q  [\frac{Z_{U^*(v)}}{Z_v} (C(U^*(v)) - C_v) \vert \f_v ], Z \in \z \rbrace$ is closed under pairwise maximization. Then, once more the fundamental property of the essential infimum/supremum guarantees that there is a sequence $(Z^n)_{n\in \mathbb{N}} \subseteq \z$ such that a.s.
 \begin{equation*}
     \essinf_{ \p \in \pp}  E^\p [C_{U^*(v)} - C_v \vert \f_v]=  \lim\limits_{n \rightarrow \infty} \downarrow E^\q \Big [\frac{Z^n_{U^*(v)}}{Z^n_v} (C_{U^*(v)} - C_v) \Big\vert \f_v \Big]=0.
 \end{equation*} 
 Using Fatou's lemma, we get 
 \begin{align*}
     0 &\leq E^\q \Big[ \lim\limits_{n \rightarrow \infty} \downarrow  \frac{Z^n_{U^*(v)}}{Z^n_v} \cdot (C_{U^*(v)} - C_v) \Big\vert \f_v \Big] \\ &\leq \lim\limits_{n \rightarrow \infty} \downarrow E^\q \Big [\frac{Z^n_{U^*(v)}}{Z^n_v} (C_{U^*(v)} - C_v) \Big\vert \f_v \Big] =0.
 \end{align*}
Thus, 
\begin{equation} \label{E C-C=0}
    E^\q \Big[ \lim\limits_{n \rightarrow \infty} \downarrow  \frac{Z^n_{U^*(v)}}{Z^n_v} \cdot (C_{U^*(v)} - C_v) \Big\vert \f_v \Big] \equiv 0.
\end{equation}
Now, since $\mathcal{H}$ is closed in $\mathbf{H}^2$, it follows from \cite[Lemma 2.22.]{Zam} that, for all $\tau,\sigma \in \s$,
\begin{equation}
    \essinf_{Z \in \z} \frac{Z_\sigma}{Z_\tau} >0 \quad \text{a.s.}
\end{equation}
Therefore, $\lim\limits_{n \rightarrow \infty} \downarrow  \frac{Z^n_{U^*(v)}}{Z^n_v} >0$, and by \eqref{E C-C=0} we deduce that $C_{U^*(v)} = C_v$ a.s., since the process $C$ is increasing. By equation \eqref{decomp of R}, this leads immediately to
\begin{equation*}
    E^\p [R(v)]= X_0 + E^\p [M_v]- E^\p[C_v]= X_0 + E^\p [M_{U^*(v)}]- E^\p[C_{U^*(v)}]= E^\p [R(U^*(v))],
\end{equation*}
from which we conclude that $E^\p [Y(U^*(v)) \vert \f_v]=E^\p [R(U^*(v)) \vert \f_v]=  R(v)$, a.s. Hence $\p \equiv \p^*$ is an optimal model, and by arbitrariness of $\p$, any model is then optimal. 
\end{proof}

%
%

\begin{acks}[Acknowledgments]
The authors would like to acknowledge that Youssef Ouknine is also affiliated with Hassan II Academy of Sciences and Technologies, and Africa Business School, Mohammed VI Polytechnic University.
\end{acks}


\begin{thebibliography}{9}


\bibitem{knockin} Ait Sahalia, F., Imhof, L. and Lai, T. L. (2004). Pricing and Hedging of American Knock-In Options, The Journal of Derivatives Spring, 11 (3), 44-50.

\bibitem{BayKarYao} Bayraktar, E., Karatzas, I. and Yao, S. (2010). Optimal stopping for dynamic convex risk measures, Illinois J. Math., 54, 1025-1067.

\bibitem{BayYao2014} Bayraktar, E. and Yao, S. (2014). On the robust optimal stopping problem, SIAM Journal on Control and Optimization 52, 3135–3175

\bibitem{BayYaoI} Bayraktar, E. and Yao, S. (2011). Optimal stopping for non-linear expectations-Part I, Stochastic Process. Appl., 121, 185-211. 

\bibitem{BayYaoII} Bayraktar, E. and Yao, S. (2011). Optimal stopping for non-linear expectations-Part II, Stochastic Process. Appl., 121, 212-264.

\bibitem{BayYao2017} Bayraktar, E. and Yao, S. (2017). Optimal stopping with random maturity under nonlinear expectation, Stochastic Processes and their Applications 127, 2586–2629.

\bibitem{Belom} Belomestny, D. and Krätschmer, V. (2016) Optimal stopping under model uncertainty: randomized stopping times approach, The Annals of Applied Probability 26, 1260–1295.

\bibitem{g-sup systems decomposition} Bouchard, B., Possamaï, D. and Tan, X. (2016). A general Doob-Meyer-Mertens decomposition for $g$-supermartingale systems. Electron. J. Probab. 21: 1-21. 

 \bibitem{carmona} Carmona, R. (ed.) (2009). Indifference Pricing: Theory and Applications, Princeton University Press, Princeton.

\bibitem{ChengRiedel} Cheng, X. and Riedel, F. (2013). Optimal stopping under ambiguity in continuous time, Mathematics and Financial Economics 7, 29–68. 

\bibitem{CherDel} Cheridito, P., Delbaen, F. and Kupper, M. (2006). Dynamic monetary risk measures for bounded discrete-time processes, Electron. J. Probab., 11, no. 3, 57-106 (electronic).  

\bibitem{Delbaen} Delbaen, F. (2006). The structure of m-stable sets and in particular of the set of risk neutral measures, in In memoriam Paul-André Meyer: Séminaire de Probabilités XXXIX, vol. 1874 of Lecture Notes in Math., Springer, Berlin, 215-258. 

\bibitem{Probabilites et Potentiel2} Dellacherie, C. and Meyer, P.-A. (1980). Probabilités et Potentiel, Théorie des Martingales. (French) Chap. V-VIII. Nouvelle édition. Hermann, Paris. 

\bibitem{Ekren} Ekren, I., Touzi, N. and Zhang, J. (2014). Optimal stopping under nonlinear expectation, Stochastic Processes and Their Applications 124, 3277–3311. 

\bibitem{El Karoui} El Karoui, N. (1981). Les aspects probabilistes du contrôle stochastique. École d’été de Probabilités de Saint-Flour IX-1979 Lect. Notes in Math. 876, 73-238, Springer, Berlin-New York. 

\bibitem{Follmer} Föllmer, H. and Schied, A. (2004). Stochastic finance, vol. 27 of de Gruyter Studies in Mathematics, Walter de Gruyter and Co., Berlin, extended ed., An introduction in discrete time.

\bibitem{L0conv} Guo, T., Zhang, E., Wu, M., Yang, B., Yuan, G. and Zeng, X. (2016). On random convex analysis. Journal of Nonlinear and Convex Analysis 18(11). 

\bibitem{Karat} Karatzas, I. and Zamfirescu, I. M. (2005). Game approach to the optimal stopping problem, Stochastics, 77, 401-435. 

\bibitem{kobQuen} Kobylanski, M. and Quenez, M.-C. (2012). Optimal stopping time problem in a general framework, Electr. J. Prob. 17(72), 1–28.

\bibitem{kobymulti} Kobylanski, M. and Quenez, M.-C. and Rouy-Mironescu, E. (2011). Optimal multiple stopping time problem, Ann. Appl. Probab. 21, no. 4, 1365–1399.

\bibitem{Krat} Krätschmer, V., Ladkau, M., Laeven, R.J.A., Schoenmakers, J.G.M. and Stadje, M. (2018). Optimal stopping under uncertainty in drift and jump intensity. Mathematics of Operations Research 43, 1177-1209. 

\bibitem{kunitawatanabe} Kunita, H. and Watanabe, S. (1967). On square integrable martingales. Nagoya Math.
Journal, 30: 209-245.

\bibitem{Roger} Laeven, R. J. A., Schoenmakers,  J. G. M., Schweizer, N. F. F. and Stadje, M. (2021). Robust Multiple Stopping -- A Pathwise Duality Approach, Papers 2006.01802, arXiv.org, revised Sep 2021. 

\bibitem{legall} Le Gall, J.F. (2012). Mouvement Brownien, martingales et calcul stochastique. Collection Mathématiques et applications 71, Springer.

\bibitem{Maingueneau}  Maingueneau, M.-A. (1978). Temps d'arrêt optimaux et théorie générale, Séminaire de probabilités, XII (Univ. Strasbourg, Strasbourg, 1976/1977), 457--467, Lecture Notes in Math., 649, Springer, Berlin. 

\bibitem{Morlais} Morlais, M. (2013). Reflected backward stochastic differential equations and nonlinear dynamic pricing rule, Stochastics: An International Journal of Probability and Stochastic Processes, 85, 1-26. 

\bibitem{Neveu} Neveu, J. (1972). Martingales à Temps Discret. Masson et Cie, éditeurs, Paris. 

\bibitem{nutz} Nutz M. and Handel, R.V. (2013). Constructing sublinear expectations on path space, Stochastic Process. Appl. 123(8), 3100–3121. 

\bibitem{nutzzhang} Nutz, M. and Zhang, J. (2015). Optimal stopping under adverse nonlinear expectation and related games, The Annals of Applied Probability 25, 2503–2534. 

\bibitem{Riedel} Riedel, F. (2009). Optimal stopping with multiple priors, Econometrica, 77, 857-908. 

\bibitem{aguilar} Treviño-Aguilar, E. (2012). Optimal stopping under model uncertainty and the regularity of lower Snell envelopes, Quantitative Finance, 12:6, 865-871.

\bibitem{Zam} Zamfirescu, I.-M. (2003). Optimal Stopping under Model Uncertainty, Ph.D. thesis, Columbia University.


\end{thebibliography}
\end{document}